\documentclass[a4paper,12pt]{amsart}

\usepackage{amsmath,amssymb,amsthm,mathtools}
\usepackage{hyperref}
\hypersetup{
 hidelinks,
 pdftitle={Dense-Kernel and Closed-Core Reductions in the D-space Problem for Charming Spaces},
 pdfauthor={Xing-Yu Hu, Ze-Qi Yin},
 pdfsubject={General topology},
 pdfkeywords={charming space, D-space, Lindelof Sigma-space, dense kernel, closed core}
}
\usepackage{enumitem}
\usepackage[a4paper,textwidth=13.5cm,textheight=21cm,centering]{geometry}

\makeatletter
\def\@@and{\unskip,}
\makeatother

\newtheorem{theorem}{Theorem}[section]
\newtheorem{lemma}[theorem]{Lemma}
\newtheorem{proposition}[theorem]{Proposition}
\newtheorem{corollary}[theorem]{Corollary}
\newtheorem{conjecture}[theorem]{Conjecture}

\theoremstyle{definition}
\newtheorem{definition}[theorem]{Definition}
\theoremstyle{remark}
\newtheorem{remark}[theorem]{Remark}

\numberwithin{equation}{section}

\newcommand{\covK}{\operatorname{cov}_K}

\title[Dense-kernel and closed-core reductions]{Dense-kernel and closed-core reductions in the \texorpdfstring{$D$}{D}-space problem for charming spaces}
\author[X.-Y. Hu]{Xing-Yu Hu\textsuperscript{*}\kern-.15em}
\thanks{\textsuperscript{*}Corresponding author.}
\address{School of Mathematics and Statistics, Hanjiang Normal University, Shiyan, China}
\email{huxingyu@hjnu.edu.cn}

\author[Z.-Q. Yin]{Ze-Qi Yin}
\address{School of Mathematics and Statistics, Hanjiang Normal University, Shiyan, China}
\email{19371600912@163.com}
\date{}

\begin{document}

\begin{abstract}
Let $X$ be a charming space with a Lindel\"of $\Sigma$ kernel $Y$, and put $B=\overline{Y}^{\,X}$. We show that the question whether every charming space is a \(D\)-space can be reduced first to the dense-kernel case and then to a closed core. We define $\mathsf{Obs}_{cc}(Y,B)$ as the set of boundary points $x\in B\setminus Y$ such that, for every open neighborhood $U$ of $x$ in $X$, the intersection $U\cap Y$ is not countably compact, and put $H=\overline{\mathsf{Obs}_{cc}(Y,B)}^{\,B}$. We prove that $H\cap (B\setminus Y)=\mathsf{Obs}_{cc}(Y,B)$, and our main reduction theorem shows that $X$ is a \(D\)-space if and only if $H$ is. We also prove the following sufficient condition, where $\mathfrak d$ denotes the dominating number. If there is a closed set $S\subseteq B\setminus Y$ that can be covered by fewer than $\mathfrak d$ compact subsets and every point of $(B\setminus Y)\setminus S$ has an open neighborhood $U$ in $X$ for which $U\cap Y$ is countably compact, then $X$ is a \(D\)-space.
\end{abstract}
\maketitle

\section{Introduction}

The notion of a \(D\)-space goes back to van Douwen and Pfeffer \cite{vanDouwenPfeffer1979}. The question of whether every regular Lindel\"of space is a \(D\)-space remains open. See Gruenhage's survey \cite{GruenhageDsurvey} and \cite{PengHu2024} for background and a recent statement of the problem. For later work on local properties and remainders in compactifications, see also \cite{Arhangelskii2019Local}. Arhangel'ski\u{\i} introduced charming spaces in \cite{Arhangelskii2011} as a remainder-type generalization of Lindel\"of $\Sigma$ spaces, and further properties of charming spaces were studied in \cite{LiLinLin2017}. The corresponding question for charming spaces remains open.

Our principal result, Theorem~\ref{thm_cc_core_exact_reduction}, gives an exact reduction of this question to a closed core. Given a charming space $X$ with Lindel\"of $\Sigma$ kernel $Y$, we first pass to the closure $B=\overline{Y}^{\,X}$, where $Y$ is dense. Within $B$, we single out the boundary points $x\in B\setminus Y$ such that, for every open neighborhood $U$ of $x$ in $X$, the intersection $U\cap Y$ is not countably compact. These points form the \emph{countably compact obstruction set} $\mathsf{Obs}_{cc}(Y,B)$, and we put $H=\overline{\mathsf{Obs}_{cc}(Y,B)}^{\,B}$. The theorem proves that $X$ is a \(D\)-space if and only if $H$ is, while Lemma~\ref{lem_cc_core_structure} gives $H\cap (B\setminus Y)=\mathsf{Obs}_{cc}(Y,B)$. Theorem~\ref{thm_middlelayer_cc_small} handles a closed exceptional set $S\subseteq B\setminus Y$ that can be covered by fewer than $\mathfrak d$ compact subsets, provided that every point of $(B\setminus Y)\setminus S$ has an open neighborhood $U$ in $X$ for which $U\cap Y$ is countably compact.

We also derive restrictions on any closed core arising from a potential counterexample. In particular, Proposition~\ref{prop_countable_cd_closure_D} shows that the closure in $X$ of any nonempty countable closed discrete subset of $Y$ is a charming \(D\)-space, and Proposition~\ref{prop_counterexample_core_not_countably_generated} shows that a counterexample core $H$ cannot be the closure in $H$ of a countable closed discrete subset of $H\cap Y$. This motivates a conjecture on the trace $H\cap Y$. The final subsection establishes two elementary criteria for remainders and then applies known remainder theorems to topological groups and compactly-fibered coset spaces. These applications are not used in the reduction arguments. Compactly-fibered coset spaces arise in Pasynkov's theory of partial topological products \cite{Pasynkov1965}. For dichotomy-type results on quotients of topological groups and for recent work on related coset spaces, see \cite{Arhangelskii2017, ChenLingZhao2024}. For general background on topological groups and coset spaces, see \cite{ArhangelskiiTkachenko2008}.

\section{Preliminaries}

We assume familiarity with standard concepts from general topology and basic set theory. We refer to \cite{BartoszynskiJudah,Engelking} for background. For a broader account of Lindel\"of $\Sigma$ spaces, see \cite{Tkachuk2010}.

Throughout the paper, all spaces are Tychonoff.
By \cite[Theorem~3.1.8]{Engelking}, compact subspaces of a Hausdorff space are closed.

\begin{definition}
\cite{GruenhageDsurvey} A space $X$ is a \emph{\(D\)-space} if for every neighborhood assignment $\varphi\colon X\to\tau(X)$ with
$x\in\varphi(x)$ for all $x\in X$, there exists a closed discrete set $D\subseteq X$ such that
\[
X=\bigcup_{d\in D}\varphi(d).
\]
\end{definition}

By \cite[Theorem~5.2(a)]{GruenhageDsurvey}, every metrizable space is a \(D\)-space.

\begin{definition}\label{def_charming}
\cite[p.~93]{Arhangelskii2011} A space $X$ is \emph{charming} if it admits a Lindel\"of $\Sigma$ subspace $Y\subseteq X$ such that for every open neighborhood $U$ of $Y$ in $X$, the subspace $X\setminus U$ is Lindel\"of $\Sigma$.
Such a $Y$ is called a \emph{Lindel\"of $\Sigma$ kernel} of $X$.
\end{definition}

\begin{definition}\label{def_cl_co}
\cite[Definition~4.1]{LiLinLin2017} A charming space is \emph{CL charming} if it admits a closed Lindel\"of $\Sigma$ kernel.
A charming space is \emph{CO charming} if it admits a compact Lindel\"of $\Sigma$ kernel.
\end{definition}

By \cite[p.~494, Theorem]{Buzyakova2002}, every strong $\Sigma$ space is a \(D\)-space. In particular, every Lindel\"of $\Sigma$ space is a \(D\)-space. Peng \cite[Theorem~14]{Peng2010WeakMonolithic} showed that every weakly monotonically monolithic space is a \(D\)-space.

By \cite[Proposition~3.3]{Arhangelskii2011}, every charming space is Lindel\"of. We shall use this fact without further reference. In particular, if $X$ is a charming space, then every closed subspace of $X$ (such as the closure $B=\overline{Y}^{\,X}$ of any Lindel\"of $\Sigma$ kernel $Y$, or any further closed set considered in the arguments below) is Lindel\"of.

For a space $Z$, let $\covK(Z)$ denote the least cardinality of a cover of $Z$ by compact subsets of $Z$.
For $f,g\in\omega^\omega$, write $f\leq^* g$ if $f(n)\leq g(n)$ for all but finitely many $n\in\omega$.
The dominating number $\mathfrak d$ is the least cardinality of a set $\mathcal D\subseteq\omega^\omega$ such that for every $f\in\omega^\omega$ there exists $g\in\mathcal D$ with $f\leq^* g$.
No countable family is dominating. Indeed, if $\{g_k:k\in\omega\}\subseteq\omega^\omega$, then the function
\[
f(n)=\max_{k\leq n}g_k(n)+1
\]
satisfies $f\not\leq^*g_k$ for every $k\in\omega$. Hence $\mathfrak d>\omega$.

We shall use the identity
\[
\covK(\omega^\omega)=\mathfrak d.
\]
Indeed, every compact subset $K$ of $\omega^\omega$ has finite coordinate projections and is therefore bounded pointwise by some $g_K\in\omega^\omega$.
Thus, if $\mathcal C$ is a compact cover of $\omega^\omega$, then $\{g_K:K\in\mathcal C\}$ is a dominating family. Hence $\mathfrak d\leq|\mathcal C|$, and therefore $\mathfrak d\leq\covK(\omega^\omega)$.
Conversely, let $\mathcal D\subseteq\omega^\omega$ be a dominating family of cardinality $\mathfrak d$.
For $g\in\mathcal D$ and $s\in\omega^{<\omega}$, put
\[
K_{g,s}=\{f\in\omega^\omega:f\mathbin{\upharpoonright}|s|=s\text{ and }f(n)\leq g(n)\text{ whenever }n\geq |s|\}.
\]
The set $K_{g,s}$ is compact, since it is a product of singleton factors for $n<|s|$ and finite discrete factors for $n\geq |s|$.
If $f\in\omega^\omega$, choose $g\in\mathcal D$ and $N\in\omega$ such that $f(n)\leq g(n)$ whenever $n\geq N$. Then $f\in K_{g,f\mathbin{\upharpoonright}N}$, so these sets cover $\omega^\omega$.
Since $\omega^{<\omega}$ is countable and $\mathfrak d>\omega$, this is a compact cover of cardinality at most $\mathfrak d$.
Therefore $\covK(\omega^\omega)\leq\mathfrak d$.

\begin{definition}\label{def_Menger}
\cite{Hurewicz1925} A space $X$ is \emph{Menger} if for every sequence $(\mathcal U_n)_{n\in\omega}$ of open covers of $X$
there exist finite families $\mathcal V_n\subseteq \mathcal U_n$ such that
$X=\bigcup_{n\in\omega}\bigcup \mathcal V_n$.
\end{definition}

\begin{theorem}\label{thm_MengerD}
\cite[Corollary~2.7]{Aurichi2010} Every Menger space is a \(D\)-space.
\end{theorem}

\begin{lemma}\label{lem_lt_d_compact_Menger}
\cite[Lemma~13]{Tall2011} If $X$ is Lindel\"of and $X$ is the union of $<\mathfrak d$ compact subspaces, then $X$ is Menger.
\end{lemma}

It follows from \cite[p.~2576, after Proposition~1.2]{KubisOkunevSzeptycki2006} that, in Hausdorff spaces, the class of Lindel\"of $\Sigma$ spaces is closed under countable unions and under closed subspaces.

Let $\mathcal{K}(X)$ denote the family of all nonempty compact subsets of a space $X$.
A map $\Phi\colon M\to\mathcal{K}(X)$ is \emph{upper semicontinuous}, or \emph{USCO} for short,
if for every open $U\subseteq X$ the set
\[
\{m\in M:\ \Phi(m)\subseteq U\}
\]
is open in $M$.

We use the following nonempty-valued reformulation of \cite[Proposition~1.2(a) and~(c)]{KubisOkunevSzeptycki2006}.
\begin{lemma}\label{lem_Lsig_USCO}
A space $X$ is Lindel\"of $\Sigma$ if and only if there exist a second countable Tychonoff space $M$
and a USCO map $\Phi\colon M\to\mathcal{K}(X)$ such that
\[
X=\bigcup_{m\in M}\Phi(m).
\]
\end{lemma}

\begin{proof}
In the cited source, all spaces are Tychonoff, but a compact-valued upper semicontinuous multivalued map may have empty values.
For the forward implication, condition~(c) of the cited proposition supplies a second countable Tychonoff space $M$ and a compact-valued upper semicontinuous multivalued map $p\colon M\to X$ such that $p(M)=X$.
Put
\[
M_0=\{m\in M:p(m)\neq\varnothing\}.
\]
The set $M\setminus M_0=\{m\in M:p(m)\subseteq\varnothing\}$ is open by upper semicontinuity, so $M_0$ is closed in $M$ and is therefore second countable and Tychonoff.
The restriction of $p$ to $M_0$ is a USCO map into $\mathcal{K}(X)$ whose values cover $X$.
Conversely, any USCO map with the stated properties satisfies condition~(c) of the cited proposition, so $X$ is Lindel\"of $\Sigma$.
\end{proof}

The next compactness lemma turns compact subsets of the parameter space into compact subsets of the represented space.

\begin{lemma}\label{lem_usco_compact_union}
Let $\Phi\colon M\to\mathcal{K}(X)$ be a USCO map and let $K\subseteq M$ be compact.
Then the saturation
\[
\Phi[K]=\bigcup_{m\in K}\Phi(m)
\]
is compact in $X$.
\end{lemma}

\begin{proof}
Let $\mathcal U$ be an open cover of $\Phi[K]$.
For each $m\in K$, choose a finite family $\mathcal U_m\subseteq\mathcal U$ such that
\[
\Phi(m)\subseteq W_m=\bigcup\mathcal U_m.
\]
By upper semicontinuity, the set
\[
O_m=\{t\in M:\Phi(t)\subseteq W_m\}
\]
is an open neighborhood of $m$ in $M$.
Compactness of $K$ gives $m_1,\ldots,m_r\in K$ such that
\[
K\subseteq O_{m_1}\cup\cdots\cup O_{m_r}.
\]
Then
\[
\Phi[K]\subseteq W_{m_1}\cup\cdots\cup W_{m_r},
\]
so $\mathcal U$ has a finite subcover.
\end{proof}

\begin{corollary}\label{cor_usco_covK}
If
\[
X=\bigcup_{m\in M}\Phi(m)
\]
for a USCO map $\Phi\colon M\to\mathcal{K}(X)$, then
\[
\covK(X)\leq\covK(M).
\]
Consequently, if $M$ is $\sigma$-compact, then $X$ is $\sigma$-compact.
\end{corollary}

\begin{proof}
Let $\{K_i:i\in I\}$ be a compact cover of $M$.
By Lemma~\ref{lem_usco_compact_union}, each $\Phi[K_i]$ is compact in $X$, and
\[
X=\bigcup_{i\in I}\Phi[K_i].
\]
Hence $\covK(X)\leq |I|$.
Taking the minimum over all compact covers of $M$ gives
\[
\covK(X)\leq\covK(M).
\]
If $M$ is $\sigma$-compact, take $I=\omega$ and obtain that $X$ is $\sigma$-compact.
\end{proof}

For comparison with the closures used below, we note that $\covK$ need not decrease when a set is replaced by its closure, even in Hausdorff spaces.

\begin{remark}\label{rem_covK_not_closure_monotone}
The inequality $\covK(\overline{A}^{\,Z})\le\covK(A)$ need not hold even for Hausdorff spaces. Let $Z=\omega^\omega$ and let $A$ be a countable dense subset of $Z$. Then $\covK(A)\le\omega$, while the identity above gives $\covK(Z)=\mathfrak d>\omega$.
\end{remark}

\begin{definition}\label{def_LC-locus}
For a space $X$, an open set $V\subseteq X$ is \emph{relatively compact} if its closure $\overline{V}^{\,X}$ is compact.
The \emph{locally compact locus} of $X$ is
\[
\mathrm{LC}(X)=\{x\in X\colon x\text{ has a relatively compact open neighborhood in }X\}.
\]
\end{definition}

The locally compact locus $\mathrm{LC}(X)$ is open in $X$. Indeed, if $x\in\mathrm{LC}(X)$ and $V$ is a relatively compact open neighborhood of $x$, then every point of $V$ has $V$ as a relatively compact open neighborhood, so $V\subseteq \mathrm{LC}(X)$.

\begin{definition}\label{def_NLC}
For a space $X$ we write
\[
\mathrm{NLC}(X)=X\setminus \mathrm{LC}(X).
\]
\end{definition}

We shall use the following standard fact.

\begin{lemma}\label{lem_lc_lindelof_sigmacompact}
If $Z$ is Hausdorff, locally compact and Lindel\"of, then $Z$ is $\sigma$-compact.
\end{lemma}

\begin{proof}
For each $z\in Z$ choose an open neighborhood $V_z$ with compact closure in $Z$.
The family $\{V_z:z\in Z\}$ is an open cover of $Z$, hence has a countable subcover
$\{V_{z_n}:n\in\omega\}$.
Then $Z=\bigcup_{n\in\omega}\overline{V_{z_n}}^{\,Z}$ is a countable union of compact sets.
\end{proof}

We shall repeatedly use the following elementary union lemma for closed discrete sets.

\begin{lemma}\label{lem_disjoint_cd_union}
Let $X$ be a topological space and let $D_1,D_2\subseteq X$ be disjoint sets, each closed discrete in $X$. Then $D_1\cup D_2$ is closed discrete in $X$.
\end{lemma}

\begin{proof}
The union $D_1\cup D_2$ is closed in $X$ as a union of two closed sets.
For each $x\in D_1$, choose an open set $O_x$ in $X$ with $O_x\cap D_1=\{x\}$. Since $D_2$ is closed in $X$ and $x\notin D_2$, the set $O_x\setminus D_2$ is an open neighborhood of $x$ whose intersection with $D_1\cup D_2$ is $\{x\}$.
The symmetric argument handles $x\in D_2$. Hence $D_1\cup D_2$ is discrete in $X$, and therefore closed discrete in $X$.
\end{proof}

\section{Main results}\label{sec_main}

\subsection{Basic reductions}

We use a gluing lemma for \(D\)-spaces.

\begin{lemma}\label{lem_glue}
Let $X$ be a space and let $B\subseteq X$ be closed.
Assume that $B$ is a \(D\)-space and that for every open set $U\subseteq X$ with $B\subseteq U$,
the closed subspace $X\setminus U$ is a \(D\)-space.
Then $X$ is a \(D\)-space.
\end{lemma}

\begin{proof}
Let $\varphi\colon X\to\tau(X)$ be a neighborhood assignment.
Restrict $\varphi$ to $B$ by $\varphi_B(b)=\varphi(b)\cap B$.
Choose a closed discrete $D_B\subseteq B$ with
\[
B\subseteq\bigcup_{d\in D_B}\varphi_B(d)\subseteq\bigcup_{d\in D_B}\varphi(d).
\]
Put $U=\bigcup_{d\in D_B}\varphi(d)$ and let $A=X\setminus U$.
Then $A$ is closed in $X$ and, by assumption, a \(D\)-space. Since $B\subseteq U$, we have
\[
A\cap B=\varnothing.
\]
Restrict $\varphi$ to $A$ by $\varphi_A(a)=\varphi(a)\cap A$.
Choose a closed discrete $D_A\subseteq A$ with
\[
A\subseteq\bigcup_{a\in D_A}\varphi_A(a)\subseteq\bigcup_{a\in D_A}\varphi(a).
\]
Since $D_B$ is closed in the closed subspace $B$ of $X$, it is closed in $X$. Similarly, $D_A$ is closed in $X$. It remains to verify discreteness in $X$. Given $x\in D_B$, pick an open set $U_x$ in $B$ with $U_x\cap D_B=\{x\}$ and write $U_x=B\cap O_x$ for some open set $O_x$ in $X$. Since $D_B\subseteq B$, we have $O_x\cap D_B=\{x\}$, so $D_B$ is discrete in $X$. The same argument applies to $D_A$. Hence both $D_B$ and $D_A$ are closed discrete in $X$.
Since $D_B\subseteq B$ and $D_A\subseteq A$ with $A\cap B=\varnothing$, the sets $D_B$ and $D_A$ are disjoint.
By Lemma~\ref{lem_disjoint_cd_union}, $D_B\cup D_A$ is closed discrete in $X$. Since
\[
X\subseteq U\cup A\subseteq \bigcup_{d\in D_B}\varphi(d)\cup \bigcup_{a\in D_A}\varphi(a),
\]
it follows that $X$ is a \(D\)-space.
\end{proof}

\begin{proposition}\label{prop_reduction_charming}
Let $X$ be a charming space and let $Y$ be a Lindel\"of $\Sigma$ kernel of $X$.
Put $B=\overline{Y}^{\,X}$.
If $B$ is a \(D\)-space, then $X$ is a \(D\)-space.
\end{proposition}

\begin{proof}
Let $U$ be an open subset of $X$ with $B\subseteq U$.
Then $Y\subseteq U$, so $X\setminus U$ is Lindel\"of $\Sigma$ by Definition~\ref{def_charming}.
By \cite[p.~494, Theorem]{Buzyakova2002}, $X\setminus U$ is a \(D\)-space.
The conclusion follows from Lemma~\ref{lem_glue}.
\end{proof}

\begin{corollary}\label{cor_metrizable_kernel_closure}
Let $X$ be a charming space with a Lindel\"of $\Sigma$ kernel $Y$, and put $B=\overline{Y}^{\,X}$.
If $B$ is metrizable, then $X$ is a \(D\)-space.
\end{corollary}

\begin{proof}
By \cite[Theorem~5.2(a)]{GruenhageDsurvey}, every metrizable space is a \(D\)-space, so $B$ is a \(D\)-space.
The conclusion follows from Proposition~\ref{prop_reduction_charming}.
\end{proof}

Closed kernels are handled directly, since Lindel\"of $\Sigma$ spaces are \(D\)-spaces.

\begin{corollary}\label{cor_closed_kernel_D}
Every CL charming space is a \(D\)-space.
\end{corollary}

\begin{proof}
Let $X$ be CL charming and let $Y$ be a closed Lindel\"of $\Sigma$ kernel of $X$.
The space $Y$ is Lindel\"of $\Sigma$. By \cite[p.~494, Theorem]{Buzyakova2002}, it is a \(D\)-space.
If $U$ is open in $X$ and $Y\subseteq U$, then $X\setminus U$ is Lindel\"of $\Sigma$ by Definition~\ref{def_charming}. By \cite[p.~494, Theorem]{Buzyakova2002}, it is a \(D\)-space.
The conclusion follows from Lemma~\ref{lem_glue} with $B=Y$.
\end{proof}

In particular, every CO charming space is a \(D\)-space, since every compact kernel is closed.

Corollary~\ref{cor_closed_kernel_D} shows that the main difficulty arises when kernels are not closed. A closed discrete set in the kernel need not be closed in the ambient space, and the basic obstruction comes from accumulation at points of $\overline{Y}^{\,X}\setminus Y$. For example, in $X=\mathbb R$ with $Y=(0,1)$, the set $\{1/(n+2)\colon n\in\omega\}$ is closed and discrete in $Y$ but not closed in $X$.

\medskip
We now consider the size of the closure $B=\overline{Y}^{\,X}$.

\begin{theorem}\label{thm_closure-covK}
Let $X$ be a charming space with a Lindel\"of $\Sigma$ kernel $Y$, and put $B=\overline{Y}^{\,X}$.
If $\covK(B)<\mathfrak d$, then $X$ is a \(D\)-space.
\end{theorem}

\begin{proof}
The closed subspace $B$ of the Lindel\"of space $X$ is Lindel\"of.
If $\covK(B)<\mathfrak d$, then $B$ is the union of $<\mathfrak d$ compact subspaces.
By Lemma~\ref{lem_lt_d_compact_Menger}, the space $B$ is Menger.
By Theorem~\ref{thm_MengerD}, the space $B$ is a \(D\)-space.
The conclusion follows from Proposition~\ref{prop_reduction_charming}.
\end{proof}

Theorem~\ref{thm_closure-covK} handles several basic cases. More generally, since
\[
B=Y\cup(B\setminus Y),
\]
one has
\[
\covK(B)\le \covK(Y)+\covK(B\setminus Y).
\]
Hence, if $\covK(Y)<\mathfrak d$ and $\covK(B\setminus Y)<\mathfrak d$, then $\covK(B)<\mathfrak d$, so $X$ is a \(D\)-space by Theorem~\ref{thm_closure-covK}. In particular, the same conclusion holds whenever $|B|<\mathfrak d$, or whenever $B$ is $\sigma$-compact.

\begin{lemma}\label{lem_lc_part_kernel_closure}
Let $X$ be a charming space with a Lindel\"of $\Sigma$ kernel $Y$ and put $B=\overline{Y}^{\,X}$.
If $L=B\cap \mathrm{LC}(X)$ is Lindel\"of, then $L$ is $\sigma$-compact.
\end{lemma}

\begin{proof}
We first observe that the open subspace $\mathrm{LC}(X)$ of the Tychonoff space $X$ is itself locally compact and Hausdorff.
Indeed, fix $x\in\mathrm{LC}(X)$ and choose an open neighborhood $W$ of $x$ in $X$ with $\overline{W}^{\,X}$ compact.
Since $\mathrm{LC}(X)$ is open in $X$ and $X$ is regular, we may shrink $W$ so that $\overline{W}^{\,X}\subseteq \mathrm{LC}(X)$. Then $\overline{W}^{\,X}$ is a compact neighborhood of $x$ in $\mathrm{LC}(X)$.

Since $B$ is closed in $X$, the set $L=B\cap \mathrm{LC}(X)$ is closed in $\mathrm{LC}(X)$, and hence $L$ is locally compact and Hausdorff.
By hypothesis, $L$ is Lindel\"of, so Lemma~\ref{lem_lc_lindelof_sigmacompact} implies that $L$ is $\sigma$-compact.
\end{proof}

\smallskip
The next proposition handles the case where the boundary is contained in an $F_\sigma$ subset of $B$ that lies in the locally compact locus, without using compact-covering arguments.

\begin{proposition}\label{prop_LC_fsigma}
Let $X$ be a charming space with a Lindel\"of $\Sigma$ kernel $Y$ and put $B=\overline{Y}^{\,X}$.
Assume that there are closed sets $F_n\subseteq B$ for $n\in\omega$ such that
$B\setminus Y\subseteq \bigcup_{n\in\omega}F_n\subseteq \mathrm{LC}(B)$.
Then $X$ is a \(D\)-space.
\end{proposition}

\begin{proof}
The closed subspace $B$ is Lindel\"of. As a subspace of the Tychonoff space $X$, the space $B$ is Hausdorff. Put
\[
Z_0=\bigcup_{n\in\omega}F_n.
\]
For each $n\in\omega$, the set $F_n$ is closed in $B$, hence Lindel\"of. Fix $x\in F_n$. Since $x\in \mathrm{LC}(B)$, there exists an open set $U_x\subseteq B$ such that
\[
x\in U_x
\quad\text{and}\quad
\overline{U_x}^{\,B}\text{ is compact.}
\]
Then $U_x\cap F_n$ is open in $F_n$. The subspace closure formula gives
\[
\overline{U_x\cap F_n}^{\,F_n}
=
F_n\cap \overline{U_x\cap F_n}^{\,B}
\subseteq
F_n\cap \overline{U_x}^{\,B}.
\]
Since $F_n$ is closed in $B$ and $\overline{U_x}^{\,B}$ is compact, the set $F_n\cap \overline{U_x}^{\,B}$ is compact. The set $\overline{U_x\cap F_n}^{\,F_n}$ is a closed subset of $F_n\cap \overline{U_x}^{\,B}$, so it is compact. Thus $F_n$ is locally compact. By Lemma~\ref{lem_lc_lindelof_sigmacompact}, each $F_n$ is $\sigma$-compact. Hence $Z_0$ is $\sigma$-compact. By \cite[p.~2576, after Proposition~1.2]{KubisOkunevSzeptycki2006}, $Z_0$ is Lindel\"of $\Sigma$. Since $B\setminus Y\subseteq Z_0$, we have $B=Y\cup Z_0$. The spaces $Y$ and $Z_0$ are Lindel\"of $\Sigma$, and $B$ is Hausdorff. By \cite[p.~2576, after Proposition~1.2]{KubisOkunevSzeptycki2006}, $B$ is Lindel\"of $\Sigma$. By \cite[p.~494, Theorem]{Buzyakova2002}, $B$ is a \(D\)-space. Proposition~\ref{prop_reduction_charming} then gives that $X$ is a \(D\)-space.
\end{proof}

\begin{corollary}\label{cor_nlc_covK_small}
Let $X$ be a charming space with a Lindel\"of $\Sigma$ kernel $Y$ and put $B=\overline{Y}^{\,X}$.
Put $L=B\cap \mathrm{LC}(X)$ and $S=B\setminus \mathrm{LC}(X)$.
If $L$ is Lindel\"of and $\covK(S)<\mathfrak d$, then $X$ is a \(D\)-space.
\end{corollary}

\begin{proof}
Since $L$ is Lindel\"of, Lemma~\ref{lem_lc_part_kernel_closure} shows that $L$ is $\sigma$-compact, so $\covK(L)\le \omega$.
Because $B=L\cup S$,
\[
\covK(B)\le \covK(L)+\covK(S)\le \omega+\covK(S)=\max\{\omega,\covK(S)\}<\mathfrak d.
\]
The conclusion follows from Theorem~\ref{thm_closure-covK}.
\end{proof}

Since $S=B\setminus \mathrm{LC}(X)=B\cap \mathrm{NLC}(X)$ is closed in $\mathrm{NLC}(X)$, every compact cover of $\mathrm{NLC}(X)$ restricts to a compact cover of $S$. Hence $\covK(S)<\mathfrak d$ whenever $\covK(\mathrm{NLC}(X))<\mathfrak d$. In particular, $\covK(S)<\mathfrak d$ whenever $B\setminus \mathrm{LC}(X)$ is $\sigma$-compact.

\begin{proposition}\label{prop_local_closed_lc_decomp}
Let $X$ be a charming space with a Lindel\"of $\Sigma$ kernel $Y$ and put $B=\overline{Y}^{\,X}$.
Assume that
\[
B=L\cup S,
\]
where $L$ and $S$ are closed in $B$, the space $L$ is locally compact, and $\covK(S)<\mathfrak d$.
Then $X$ is a \(D\)-space.
\end{proposition}

\begin{proof}
The closed subspace $B$ of $X$ is Lindel\"of.
The closed subspace $L$ of $B$ is therefore Lindel\"of.
Moreover $L$ is Hausdorff as a subspace of the Tychonoff space $X$.
By assumption, the space $L$ is locally compact.
Thus Lemma~\ref{lem_lc_lindelof_sigmacompact} shows that $L$ is $\sigma$-compact.
Therefore $\covK(L)\le \omega<\mathfrak d$.
Since $B=L\cup S$, we have
\[
\covK(B)\le \covK(L)+\covK(S)\le \omega+\covK(S)=\max\{\omega,\covK(S)\}<\mathfrak d.
\]
The conclusion now follows from Theorem~\ref{thm_closure-covK}.
\end{proof}

\smallskip

\begin{proposition}\label{prop_Gdelta_kernel}
Let $X$ be a charming space with a Lindel\"of $\Sigma$ kernel $Y$ and put $B=\overline{Y}^{\,X}$.
If $Y$ is a $G_\delta$ subset of $B$, then $X$ is a \(D\)-space.
\end{proposition}

\begin{proof}
Choose open sets $V_n$ in $B$ for $n\in\omega$ such that $Y=\bigcap_{n\in\omega}V_n$.
For each $n\in\omega$ pick an open set $U_n$ in $X$ with $V_n=U_n\cap B$.
Since $Y\subseteq V_n\subseteq U_n$ for each $n\in\omega$ and $X$ is charming, the space $X\setminus U_n$ is Lindel\"of $\Sigma$.
The set $B\setminus V_n=B\cap (X\setminus U_n)$ is closed in $X\setminus U_n$. By \cite[p.~2576, after Proposition~1.2]{KubisOkunevSzeptycki2006}, $B\setminus V_n$ is Lindel\"of $\Sigma$.
Now
\[
B=Y\cup \bigcup_{n\in\omega}(B\setminus V_n).
\]
The subspace $Y$ is Lindel\"of $\Sigma$, and each $B\setminus V_n$ is Lindel\"of $\Sigma$. By \cite[p.~2576, after Proposition~1.2]{KubisOkunevSzeptycki2006}, since $B$ is Hausdorff as a subspace of the Tychonoff space $X$, $B$ is Lindel\"of $\Sigma$.
By \cite[p.~494, Theorem]{Buzyakova2002}, $B$ is a \(D\)-space.
Apply Proposition~\ref{prop_reduction_charming}.
\end{proof}

\begin{proposition}\label{prop_Gdelta_kernel_ambient}
Let $X$ be a charming space with a Lindel\"of $\Sigma$ kernel $Y$.
If $Y$ is a $G_\delta$ subset of $X$, then $X$ is Lindel\"of $\Sigma$.
Consequently $X$ is a \(D\)-space.
\end{proposition}

\begin{proof}
Choose open sets $U_n$ in $X$ for $n\in\omega$ such that $Y=\bigcap_{n\in\omega}U_n$.
Then
\[
X=Y\cup \bigcup_{n\in\omega}\bigl(X\setminus U_n\bigr).
\]
By Definition~\ref{def_charming}, the kernel $Y$ is Lindel\"of $\Sigma$ and each $X\setminus U_n$ is Lindel\"of $\Sigma$. By \cite[p.~2576, after Proposition~1.2]{KubisOkunevSzeptycki2006}, since $X$ is Hausdorff, the displayed decomposition shows that $X$ is Lindel\"of $\Sigma$.
By \cite[p.~494, Theorem]{Buzyakova2002}, $X$ is a \(D\)-space.
\end{proof}

\medskip
The preceding arguments also give several direct criteria.

\begin{proposition}\label{prop_kernel_regular}
Let $X$ be a charming space with a Lindel\"of $\Sigma$ kernel $Y$ and put $B=\overline{Y}^{\,X}$.
Assume that at least one of the following holds.
\begin{enumerate}[label=\arabic*.]
\item $Y$ is a $G_\delta$ subset of $B$.
Equivalently, there exists a sequence of open sets $(U_n)_{n\in\omega}$ in $B$ such that $Y=\bigcap_{n\in\omega}U_n$.
\item $B\setminus Y$ is an $F_\sigma$ subset of $B$.
In particular, this holds when $B\setminus Y$ is $\sigma$-compact in $B$.
\item $B\setminus Y$ is Lindel\"of and $B\setminus Y\subseteq \mathrm{LC}(B)$.
\item The set $B\setminus Y$ is contained in a locally compact $F_\sigma$ subspace of $B$.
\end{enumerate}
Then $X$ is a \(D\)-space.
\end{proposition}

\begin{proof}
If item 1 holds, the conclusion follows from Proposition~\ref{prop_Gdelta_kernel}.
If item 2 holds, then $Y$ is a $G_\delta$ subset of $B$, so the conclusion follows from Proposition~\ref{prop_Gdelta_kernel}.
Assume item 3 holds. If $B\setminus Y=\varnothing$, then item 2 applies. Suppose that $B\setminus Y\neq\varnothing$.
For each $x\in B\setminus Y$, choose an open neighborhood $V_x$ of $x$ in $B$ such that $\overline{V_x}^{\,B}$ is compact.
Since $\mathrm{LC}(B)$ is open in the regular space $B$, choose an open neighborhood $W_x$ of $x$ in $B$ such that
\[
\overline{W_x}^{\,B}\subseteq V_x\cap\mathrm{LC}(B).
\]
Then $\overline{W_x}^{\,B}$ is a closed subset of the compact set $\overline{V_x}^{\,B}$, and is therefore compact.
The family
\[
\{W_x\cap(B\setminus Y):x\in B\setminus Y\}
\]
is an open cover of the Lindel\"of space $B\setminus Y$.
Choose points $x_n\in B\setminus Y$ for $n\in\omega$ such that
\[
B\setminus Y\subseteq\bigcup_{n\in\omega}W_{x_n}.
\]
For each $n\in\omega$, put $F_n=\overline{W_{x_n}}^{\,B}$.
Then every $F_n$ is compact and closed in $B$, and
\[
B\setminus Y\subseteq\bigcup_{n\in\omega}F_n\subseteq\mathrm{LC}(B).
\]
Proposition~\ref{prop_LC_fsigma} now gives the conclusion.

Assume item 4 holds, and choose a locally compact $F_\sigma$ subspace $A$ of $B$ such that $B\setminus Y\subseteq A$.
Write $A=\bigcup_{n\in\omega}C_n$, where each $C_n$ is closed in $B$. Since $B$ is Lindel\"of, every $C_n$ is Lindel\"of, and therefore $A$ is Lindel\"of.
As a subspace of the Tychonoff space $B$, the space $A$ is Hausdorff.
Lemma~\ref{lem_lc_lindelof_sigmacompact} shows that $A$ is $\sigma$-compact, and hence $A$ is Lindel\"of $\Sigma$.
Since $B=Y\cup A$, \cite[p.~2576, after Proposition~1.2]{KubisOkunevSzeptycki2006} shows that $B$ is Lindel\"of $\Sigma$. By \cite[p.~494, Theorem]{Buzyakova2002}, the space $B$ is a \(D\)-space. Proposition~\ref{prop_reduction_charming} then shows that $X$ is a \(D\)-space.
\end{proof}

\smallskip
\noindent\textit{Restrictions on counterexamples.}\par

Combining the preceding results gives the following restrictions on any non-\(D\) charming space.

\begin{proposition}\label{prop_countertemplate}
Let $X$ be a charming space that is not a \(D\)-space.
Let $Y$ be a Lindel\"of $\Sigma$ kernel of $X$ and put $B=\overline{Y}^{\,X}$.
Then the following hold.
\begin{enumerate}[label=\textup{(\arabic*)}]
\item $Y$ is not closed in $X$.
\item $\covK(B)\ge \mathfrak d$.
\item $Y$ is not a $G_\delta$ subset of $B$.
\item The set $B\setminus Y$ is not contained in any locally compact $F_\sigma$ subspace of $B$.
\item The set $B\setminus Y$ is not contained in any $F_\sigma$ subset of $B$ that lies in $\mathrm{LC}(B)$.
\end{enumerate}
\end{proposition}

\begin{proof}
Item (1) follows by contraposition from Corollary~\ref{cor_closed_kernel_D}.
Item (2) follows by contraposition from Theorem~\ref{thm_closure-covK}.
If item (3) did not hold, then Proposition~\ref{prop_Gdelta_kernel} would imply that $X$ is a \(D\)-space.
If item (4) did not hold, then item 4 of Proposition~\ref{prop_kernel_regular} would imply that $X$ is a \(D\)-space.
If item (5) did not hold, then Proposition~\ref{prop_LC_fsigma} would imply that $X$ is a \(D\)-space.
\end{proof}

Thus, in ZFC, every counterexample must fall outside all of the preceding cases. In particular, if $B=\overline{Y}^{\,X}$, then $\covK(B)\ge\mathfrak d$, $Y$ is not a $G_\delta$ subset of $B$, and $B\setminus Y$ is neither contained in a locally compact $F_\sigma$ subspace of $B$ nor in an $F_\sigma$ subset of $B$ that lies in $\mathrm{LC}(B)$.

Together with the fact that $B$ is Lindel\"of, the inequality $\covK(B)\ge\mathfrak d$ has two further consequences.
First, $B$ cannot be locally compact. Otherwise, being Lindel\"of, locally compact and Hausdorff, it would be $\sigma$-compact by Lemma~\ref{lem_lc_lindelof_sigmacompact}, contradicting $\covK(B)\ge \mathfrak d$. Second, $B$ is not contained in $\mathrm{LC}(X)$. Otherwise $B$ would be a closed subspace of the locally compact Hausdorff space $\mathrm{LC}(X)$ and hence itself locally compact, which gives the same contradiction.

Hence, after passing to the closure of a kernel, the search for a counterexample may be restricted to the dense-kernel case. Any such counterexample must evade both the compact-covering bounds below $\mathfrak d$ and the criteria based on the locally compact locus.

\subsection{Dense-kernel reformulation}

\begin{lemma}\label{lem_closure_charming}
Let $X$ be a charming space with a Lindel\"of $\Sigma$ kernel $Y$ and put $B=\overline{Y}^{\,X}$.
Then $B$ is charming and $Y$ is a dense Lindel\"of $\Sigma$ kernel of $B$.
\end{lemma}

\begin{proof}
The subspace $Y$ is Lindel\"of $\Sigma$ in $B$.
Let $V$ be an open neighborhood of $Y$ in $B$.
Choose an open set $U$ in $X$ with $V=U\cap B$.
Then $B\setminus V=B\cap (X\setminus U)$.
Since $Y\subseteq V\subseteq U$ and $X$ is charming, the space $X\setminus U$ is Lindel\"of $\Sigma$.
The set $B\setminus V$ is closed in $X\setminus U$. By \cite[p.~2576, after Proposition~1.2]{KubisOkunevSzeptycki2006}, $B\setminus V$ is Lindel\"of $\Sigma$.
Thus $Y$ is a Lindel\"of $\Sigma$ kernel of $B$.
The density of $Y$ in $B$ holds by definition of $B$.
\end{proof}

Thus the general problem reduces to the dense-kernel case.

\begin{theorem}\label{thm_equiv_dense_extension}
The following are equivalent.
\begin{enumerate}[label=\textup{(\roman*)}]
\item Every charming space is a \(D\)-space.
\item Every charming space with a dense Lindel\"of $\Sigma$ kernel is a \(D\)-space.
\end{enumerate}
\end{theorem}

\begin{proof}
(i)$\Rightarrow$(ii). A charming space with a dense Lindel\"of $\Sigma$ kernel is, in particular, a charming space. Hence it is a \(D\)-space by~(i).

(ii)$\Rightarrow$(i). Let $X$ be a charming space with a Lindel\"of $\Sigma$ kernel $Y$, and put
\[
B=\overline{Y}^{\,X}.
\]
By Lemma~\ref{lem_closure_charming}, the space $B$ is charming and $Y$ is a dense Lindel\"of $\Sigma$ kernel of $B$. Hence item (ii) applies to $B$, so $B$ is a \(D\)-space. Proposition~\ref{prop_reduction_charming} then gives that $X$ is a \(D\)-space.
\end{proof}

\subsection{Boundary cases}

\begin{lemma}\label{lem_countable_kernel_trace}
Let $X$ be a space, let $Y\subseteq X$ be Lindel\"of $\Sigma$, and let $\varphi$ be a neighborhood assignment on $X$.
Then there exists a countable set $C\subseteq Y$ such that
\[
Y\subseteq \bigcup_{c\in C}\varphi(c).
\]
\end{lemma}

\begin{proof}
Since $Y$ is Lindel\"of, the family
\[
\{\varphi(y)\cap Y:y\in Y\}
\]
has a countable subcover. Hence there is a countable set $C\subseteq Y$ such that
\[
Y\subseteq \bigcup_{c\in C}\bigl(\varphi(c)\cap Y\bigr)
\subseteq \bigcup_{c\in C}\varphi(c).
\]
\end{proof}

\begin{proof}[An alternative proof using the USCO representation]
If $Y=\varnothing$, take $C=\varnothing$. Assume that $Y\neq\varnothing$.
By Lemma~\ref{lem_Lsig_USCO}, choose a second countable \mbox{Tychonoff} space $M$ and a USCO map
$\Phi\colon M\to\mathcal{K}(Y)$ such that
\[
Y=\bigcup_{m\in M}\Phi(m).
\]
For each $m\in M$, compactness of $\Phi(m)$ gives a finite set $F_m\subseteq\Phi(m)$ such that
\[
\Phi(m)\subseteq U_m:=\bigcup_{y\in F_m}\bigl(\varphi(y)\cap Y\bigr).
\]
By upper semicontinuity,
\[
W_m=\{t\in M:\ \Phi(t)\subseteq U_m\}
\]
is an open neighborhood of $m$. Thus $\{W_m:m\in M\}$ is an open cover of $M$.
Since $M$ is second countable, choose $m_n\in M$, $n\in\omega$, such that
\[
M=\bigcup_{n\in\omega}W_{m_n}.
\]
Set
\[
C=\bigcup_{n\in\omega}F_{m_n}.
\]
Then $C$ is countable. For $y\in Y$, choose $t\in M$ with $y\in\Phi(t)$ and $n\in\omega$ with $t\in W_{m_n}$. Hence
\[
y\in\Phi(t)\subseteq U_{m_n}\subseteq\bigcup_{c\in C}\varphi(c).
\]
Therefore $Y\subseteq\bigcup_{c\in C}\varphi(c)$.
\end{proof}

\begin{corollary}\label{cor_countable_trace_charming}
Let $X$ be a charming space with a Lindel\"of $\Sigma$ kernel $Y$, and let $\varphi$ be a neighborhood assignment on $X$.
Then there exists a countable set $C\subseteq Y$ such that, for
\[
U=\bigcup_{c\in C}\varphi(c),
\]
the set $U$ is an open neighborhood of $Y$ and the complement $X\setminus U$ is Lindel\"of $\Sigma$.
\end{corollary}

\begin{proof}
By Lemma~\ref{lem_countable_kernel_trace}, there exists a countable set $C\subseteq Y$ with $Y\subseteq U$.
Thus $U$ is an open neighborhood of $Y$ in $X$.
Since $Y$ is a kernel of the charming space $X$, the complement $X\setminus U$ is Lindel\"of $\Sigma$.
\end{proof}

\begin{lemma}\label{lem_countable_cd_kernel_trace}
Let $X$ be a space, let $Y\subseteq X$ be Lindel\"of $\Sigma$, and let $\varphi$ be a neighborhood assignment on $X$.
Then there exists a countable set $D\subseteq Y$ that is closed discrete in $Y$ and satisfies
\[
Y\subseteq \bigcup_{d\in D}\varphi(d).
\]
\end{lemma}

\begin{proof}
For each $y\in Y$, put $\psi(y)=\varphi(y)\cap Y$.
Then $\psi$ is a neighborhood assignment on the subspace $Y$.
By \cite[p.~494, Theorem]{Buzyakova2002}, the Lindel\"of $\Sigma$ space $Y$ is a \(D\)-space.
Hence there exists a closed discrete set $D\subseteq Y$ such that
\[
Y\subseteq \bigcup_{d\in D}\psi(d)\subseteq \bigcup_{d\in D}\varphi(d).
\]
Because $D$ is closed in $Y$ and $Y$ is Lindel\"of, the subspace $D$ is Lindel\"of.
Since $D$ is also discrete, it must be countable.
\end{proof}

\begin{proposition}\label{prop_countable_cd_closed_criterion}
Let $X$ be a charming space with a Lindel\"of $\Sigma$ kernel $Y$.
Assume that every countable closed discrete subset of $Y$ is closed in $X$.
Then $X$ is a \(D\)-space.
\end{proposition}

\begin{proof}
Let $\varphi$ be a neighborhood assignment on $X$.
By Lemma~\ref{lem_countable_cd_kernel_trace}, there exists a countable set $D_Y\subseteq Y$ that is closed discrete in $Y$ and satisfies
\[
Y\subseteq \bigcup_{d\in D_Y}\varphi(d).
\]
By hypothesis, $D_Y$ is closed in $X$. Moreover, for each $d\in D_Y$ we may pick an open set $U_d$ in $Y$ with $U_d\cap D_Y=\{d\}$ and write $U_d=O_d\cap Y$ for some open $O_d$ in $X$. Since $D_Y\subseteq Y$, we get $O_d\cap D_Y=\{d\}$. Hence $D_Y$ is discrete, and therefore closed discrete, in $X$.

Put $U=\bigcup_{d\in D_Y}\varphi(d)$ and $A=X\setminus U$. Then $Y\subseteq U$, so $A$ is Lindel\"of $\Sigma$ by Definition~\ref{def_charming}, and hence a \(D\)-space.
Restrict $\varphi$ to $A$ by $\varphi_A(a)=\varphi(a)\cap A$, and choose a closed discrete set $D_A\subseteq A$ with
\[
A\subseteq \bigcup_{a\in D_A}\varphi_A(a)\subseteq \bigcup_{a\in D_A}\varphi(a).
\]
Since $A$ is closed in $X$, the set $D_A$ is closed in $X$. Moreover, the subspace-open-set argument used above for $D_Y$, with $A$ in place of $Y$, shows that $D_A$ is discrete in $X$, and hence $D_A$ is closed discrete in $X$.
Since $D_Y\subseteq U$ and $D_A\subseteq X\setminus U=A$, the sets $D_Y$ and $D_A$ are disjoint.
By Lemma~\ref{lem_disjoint_cd_union}, $D_Y\cup D_A$ is closed discrete in $X$. Since
\[
X\subseteq U\cup A\subseteq \bigcup_{d\in D_Y}\varphi(d)\cup\bigcup_{a\in D_A}\varphi(a),
\]
$X$ is a \(D\)-space.
\end{proof}

\begin{proposition}\label{prop_countably_closed_kernel_criterion}
Let $X$ be a charming space with a Lindel\"of $\Sigma$ kernel $Y$, and put $B=\overline{Y}^{\,X}$.
Assume that
\[
\overline{C}^{\,B}\subseteq Y
\]
for every countable set $C\subseteq Y$.
Then $X$ is a \(D\)-space.
\end{proposition}

\begin{proof}
Let $D\subseteq Y$ be countable and closed discrete in $Y$.
By assumption,
\[
\overline{D}^{\,B}\subseteq Y.
\]
Since $D$ is closed in the subspace $Y$,
\[
\overline{D}^{\,Y}=D.
\]
Because $Y$ is a subspace of $B$ and $D\subseteq Y$, the subspace closure formula gives
\[
\overline{D}^{\,Y}=Y\cap \overline{D}^{\,B}.
\]
Hence
\[
\overline{D}^{\,B}
=
\overline{D}^{\,B}\cap Y
=
\overline{D}^{\,Y}
=
D.
\]
Thus $D$ is closed in $B$, and therefore closed in $X$.
The conclusion follows from Proposition~\ref{prop_countable_cd_closed_criterion}.
\end{proof}

\begin{proposition}\label{prop_boundary_easycriteria}
Let $X$ be a charming space with a Lindel\"of $\Sigma$ kernel $Y$, and put $B=\overline{Y}^{\,X}$.
Each of the following conditions implies that $X$ is a \(D\)-space.
\begin{enumerate}[label=\textup{(\roman*)}]
\item $B\setminus Y$ is a countable union of Lindel\"of $\Sigma$ subspaces of $B$.
\item $B\setminus Y$ is an $F_\sigma$ subset of $B$.
\end{enumerate}
\end{proposition}

\begin{proof}
For item~\textup{(i)}, write
\[
B\setminus Y=\bigcup_{n\in\omega} Z_n,
\]
where each $Z_n$ is Lindel\"of $\Sigma$ in the Hausdorff space $B$.
Then
\[
B=Y\cup\bigcup_{n\in\omega} Z_n
\]
is a countable union of Lindel\"of $\Sigma$ subspaces of $B$.
By \cite[p.~2576, after Proposition~1.2]{KubisOkunevSzeptycki2006}, the space $B$ is Lindel\"of $\Sigma$.
Hence $B$ is a \(D\)-space.
The conclusion follows from Proposition~\ref{prop_reduction_charming}.

For item~\textup{(ii)}, write
\[
B\setminus Y=\bigcup_{n\in\omega} F_n,
\]
where each $F_n$ is closed in $B$.
Then
\[
Y=B\setminus \bigcup_{n\in\omega}F_n=\bigcap_{n\in\omega}(B\setminus F_n),
\]
so $Y$ is a $G_\delta$ subset of $B$.
Now Proposition~\ref{prop_Gdelta_kernel}, applied to the charming space $X$ with kernel $Y$, shows that $X$ is a \(D\)-space.
\end{proof}

\begin{corollary}\label{cor_boundary_easyshapes}
Let $X$ be a charming space with a Lindel\"of $\Sigma$ kernel $Y$, and put $B=\overline{Y}^{\,X}$.
If $B\setminus Y$ is Lindel\"of $\Sigma$, or countable, or $\sigma$-compact, or $\sigma$-closed-discrete in $B$, then $X$ is a \(D\)-space.
\end{corollary}

\begin{proof}
If $B\setminus Y$ is Lindel\"of $\Sigma$, then Proposition~\ref{prop_boundary_easycriteria}\textup{(i)} shows that $X$ is a \(D\)-space.
If $B\setminus Y$ is countable, then it is a countable union of singletons, and each singleton is closed in the $T_1$ space $B$.
Hence $B\setminus Y$ is $F_\sigma$ in $B$, so the conclusion follows from Proposition~\ref{prop_boundary_easycriteria}\textup{(ii)}.
If $B\setminus Y$ is $\sigma$-compact, write
\[
B\setminus Y=\bigcup_{n\in\omega}K_n
\]
with each $K_n$ compact in $B$.
Since $B$ is Hausdorff, every $K_n$ is closed in $B$, so $B\setminus Y$ is $F_\sigma$ in $B$.
Again the conclusion follows from Proposition~\ref{prop_boundary_easycriteria}\textup{(ii)}.
Finally, if $B\setminus Y$ is $\sigma$-closed-discrete in $B$, write
\[
B\setminus Y=\bigcup_{n\in\omega}D_n
\]
with each $D_n$ closed discrete in $B$.
Each $D_n$ is closed in $B$, hence $B\setminus Y$ is $F_\sigma$ in $B$.
So Proposition~\ref{prop_boundary_easycriteria}\textup{(ii)} shows that $X$ is a \(D\)-space as well.
\end{proof}

\begin{corollary}\label{cor_counterexample_boundary_complex}
Let $X$ be a charming space that is not a \(D\)-space.
Let $Y$ be a Lindel\"of $\Sigma$ kernel of $X$ and put $B=\overline{Y}^{\,X}$.
Then $B\setminus Y$ is neither Lindel\"of $\Sigma$ nor an $F_\sigma$ subset of $B$.
In particular, $B\setminus Y$ is uncountable, not $\sigma$-compact, and not $\sigma$-closed-discrete in $B$.
\end{corollary}

\begin{proof}
If $B\setminus Y$ were Lindel\"of $\Sigma$, then Proposition~\ref{prop_boundary_easycriteria}\textup{(i)} would imply that $X$ is a \(D\)-space.
If $B\setminus Y$ were $F_\sigma$ in $B$, then Proposition~\ref{prop_boundary_easycriteria}\textup{(ii)} would again imply that $X$ is a \(D\)-space.
If $B\setminus Y$ were countable, or $\sigma$-compact, or $\sigma$-closed-discrete in $B$, then it would be $F_\sigma$ in $B$, so the second clause would apply.
\end{proof}

\begin{proposition}\label{prop_local_countably_compact_boundary}
Let $X$ be a charming space with a Lindel\"of $\Sigma$ kernel $Y$, and put $B=\overline{Y}^{\,X}$.
Assume that every point $x\in B\setminus Y$ has an open neighborhood $U_x$ in $X$ such that
\[
U_x\cap Y
\]
is countably compact.
Then $X$ is a \(D\)-space.
\end{proposition}

\begin{proof}
Let $D\subseteq Y$ be a countable set that is closed discrete in $Y$.
We show that $D$ is closed in $X$.
If $x\in X\setminus B$, then the open set $X\setminus B$ contains $x$ and is disjoint from $D$.
If $x\in Y\setminus D$, then $x\notin \overline{D}^{\,Y}$ because $D$ is closed in $Y$.
Hence there exists an open set $U_x'$ in $Y$ such that
\[
x\in U_x'
\quad\text{and}\quad
U_x'\cap D=\varnothing.
\]
Write
\[
U_x'=V_x\cap Y
\]
for some open set $V_x$ in $X$.
Since $D\subseteq Y$, we obtain
\[
V_x\cap D=U_x'\cap D=\varnothing.
\]
Finally, let $x\in B\setminus Y$ and choose $U_x$ as in the assumption.
The set $D\cap U_x$ is countable because $D$ is countable.
For each $d\in D\cap U_x$, since $D$ is discrete in $Y$, there exists an open set $O_d$ in $Y$ such that
\[
d\in O_d
\quad\text{and}\quad
O_d\cap D=\{d\}.
\]
Then
\[
(U_x\cap Y)\cap O_d
\]
is open in $U_x\cap Y$, contains $d$, and satisfies
\[
\bigl((U_x\cap Y)\cap O_d\bigr)\cap (D\cap U_x)=\{d\}.
\]
Thus $D\cap U_x$ is discrete in $U_x\cap Y$.
Also,
\[
D\cap U_x=(U_x\cap Y)\cap D.
\]
Since $D$ is closed in $Y$ and $U_x\cap Y$ is a subspace of $Y$, the set $(U_x\cap Y)\cap D=D\cap U_x$ is closed in $U_x\cap Y$.
Thus $D\cap U_x$ is a countable closed discrete subspace of $U_x\cap Y$.
Hence $D\cap U_x$ is finite, because a countably compact $T_1$ space contains no infinite countable closed discrete subspace.
Because $D\cap U_x$ is finite and $X$ is $T_1$, the set $D\cap U_x$ is closed in $X$. Since $x\notin D\cap U_x$, there exists an open neighborhood $W_x\subseteq U_x$ of $x$ with
\[
W_x\cap (D\cap U_x)=\varnothing.
\]
Because $W_x\subseteq U_x$, this implies
\[
W_x\cap D=\varnothing.
\]
Thus no point of $X\setminus D$ belongs to $\overline{D}^{\,X}$, and therefore $D$ is closed in $X$.
The conclusion follows from Proposition~\ref{prop_countable_cd_closed_criterion}.
\end{proof}

\begin{theorem}\label{thm_middlelayer_cc_small}
Let $X$ be a charming space with a Lindel\"of $\Sigma$ kernel $Y$, and put $B=\overline{Y}^{\,X}$.
Assume that there exists a closed subset $S$ of $B$ such that $S\subseteq B\setminus Y$ and
\[
\covK(S)<\mathfrak d,
\]
and every point $x\in (B\setminus Y)\setminus S$ has an open neighborhood $U_x$ in $X$ such that
\[
U_x\cap Y
\]
is countably compact.
Then $X$ is a \(D\)-space.
\end{theorem}

\begin{proof}
The set $S$ is closed in the Lindel\"of space $B$, so $S$ is Lindel\"of.
Since $\covK(S)<\mathfrak d$, Lemma~\ref{lem_lt_d_compact_Menger} shows that $S$ is Menger, and Theorem~\ref{thm_MengerD} then implies that $S$ is a \(D\)-space.

We apply Lemma~\ref{lem_glue} to the closed subspace $S$ of $B$.
Let $U\subseteq B$ be open with $S\subseteq U$, and put
\[
F=B\setminus U
\quad\text{and}\quad
Y_F=Y\cap F.
\]
Since $F$ is closed in $B$, the space $F$ is Lindel\"of.
If $Y_F=\varnothing$, then $Y\subseteq U$, so $F=B\setminus U$ is Lindel\"of $\Sigma$ by Definition~\ref{def_charming} applied to the charming space $B$. By \cite[p.~494, Theorem]{Buzyakova2002}, $F$ is a \(D\)-space.
Suppose from now on that $Y_F\neq\varnothing$.
The set $Y_F$ is closed in the Lindel\"of $\Sigma$ space $Y$, hence Lindel\"of $\Sigma$.
We show that $Y_F$ is a Lindel\"of $\Sigma$ kernel of $F$.
Let $V$ be an open neighborhood of $Y_F$ in $F$, and write $V=W\cap F$ for some open set $W$ in $B$.
Since $Y_F\subseteq W$ and $Y\setminus F\subseteq U$, we have $Y\subseteq W\cup U$.
Since $B$ is charming with dense Lindel\"of $\Sigma$ kernel $Y$ by Lemma~\ref{lem_closure_charming}, and $W\cup U$ is an open neighborhood of $Y$ in $B$, the set $B\setminus(W\cup U)$ is Lindel\"of $\Sigma$.
Now
\[
F\setminus V=(B\setminus U)\setminus(W\cap F)=B\setminus(U\cup W),
\]
so $F\setminus V$ is Lindel\"of $\Sigma$.
Hence $Y_F$ is a Lindel\"of $\Sigma$ kernel of $F$. Therefore $F$ is charming by Definition~\ref{def_charming}.

Let $x\in F\setminus Y_F$.
Since $x\in F=B\setminus U$ and $S\subseteq U$, we have $x\notin S$. Since $x\in F$ and $x\notin Y_F=Y\cap F$, we have $x\notin Y$.
Hence $x\in (B\setminus Y)\setminus S$, so by the assumption there exists an open neighborhood $U_x$ of $x$ in $X$ such that $U_x\cap Y$ is countably compact.
The set $V_x=U_x\cap F$ is an open neighborhood of $x$ in $F$, and
\[
V_x\cap Y_F=U_x\cap Y_F=(U_x\cap Y)\cap Y_F.
\]
Since $Y_F=Y\cap F$ is closed in $Y$, the subspace $U_x\cap Y_F=(U_x\cap Y)\cap Y_F$ is closed in the countably compact space $U_x\cap Y$.
Hence $V_x\cap Y_F$ is countably compact.
Thus, for every $x\in F\setminus Y_F$, there exists an open neighborhood $V_x$ of $x$ in $F$ such that $V_x\cap Y_F$ is countably compact. In particular, this holds for every $x\in\overline{Y_F}^{\,F}\setminus Y_F\subseteq F\setminus Y_F$.
Therefore Proposition~\ref{prop_local_countably_compact_boundary}, applied to the charming space $F$ with Lindel\"of $\Sigma$ kernel $Y_F$ and closure $\overline{Y_F}^{\,F}$, shows that $F$ is a \(D\)-space.

Thus Lemma~\ref{lem_glue}, applied to the closed subspace $S$ of $B$, shows that $B$ is a \(D\)-space.
Finally Proposition~\ref{prop_reduction_charming} shows that $X$ is a \(D\)-space.
\end{proof}

\begin{corollary}\label{cor_counterexample_cd_accumulation}
Let $X$ be a charming space that is not a \(D\)-space.
Let $Y$ be a Lindel\"of $\Sigma$ kernel of $X$ and put $B=\overline{Y}^{\,X}$.
Then there exist a nonempty countable set $D\subseteq Y$ and a point $x\in B\setminus Y$ such that $D$ is closed discrete in $Y$ and
\[
x\in \overline{D}^{\,X}.
\]
In particular, there exists a countable subset of $Y$ whose closure in $B$ is not contained in $Y$.
\end{corollary}

\begin{proof}
By the contrapositive of Proposition~\ref{prop_countable_cd_closed_criterion}, there exists a countable closed discrete set $D\subseteq Y$ that is not closed in $X$. In particular, $D\neq\varnothing$.
Choose
\[
x\in \overline{D}^{\,X}\setminus D.
\]
Since $D\subseteq Y\subseteq B$ and $B$ is closed in $X$, we have
\[
\overline{D}^{\,X}\subseteq B.
\]
Because $B$ is a closed subspace of $X$ containing $D$, the subspace closure formula gives
\[
\overline{D}^{\,B}=B\cap \overline{D}^{\,X}=\overline{D}^{\,X}.
\]
If $x$ belonged to $Y\setminus D$, then, because $D\subseteq Y$ and $Y$ is a subspace of $X$, the subspace closure formula would give
\[
x\in Y\cap \overline{D}^{\,X}=\overline{D}^{\,Y}=D,
\]
a contradiction.
Hence $x\notin Y$, so in fact $x\in B\setminus Y$.
The final assertion now follows from
\[
x\in \overline{D}^{\,B}\setminus Y.
\]
\end{proof}

\begin{proposition}\label{prop_counterexample_no_local_cc}
Let $X$ be a charming space that is not a \(D\)-space.
Let $Y$ be a Lindel\"of $\Sigma$ kernel of $X$ and put $B=\overline{Y}^{\,X}$.
Then there exist a nonempty countable set $D\subseteq Y$ and a point $x\in B\setminus Y$ such that $D$ is closed discrete in $Y$, $x\in \overline{D}^{\,X}$, and for every open neighborhood $U$ of $x$ in $X$ the set
\[
U\cap D
\]
is infinite.
Consequently, for every open neighborhood $U$ of $x$ in $X$, the set $U\cap Y$ is not countably compact. Indeed, $U\cap Y$ contains the infinite countable closed discrete subspace $U\cap D$.
\end{proposition}

\begin{proof}
By the contrapositive of Proposition~\ref{prop_countable_cd_closed_criterion}, there exists a countable closed discrete set $D\subseteq Y$ that is not closed in $X$. In particular, $D\neq\varnothing$.
Choose
\[
x\in \overline{D}^{\,X}\setminus D.
\]
Since $D\subseteq Y\subseteq B$ and $B$ is closed in $X$, we have $x\in B$.
If $x$ belonged to $Y\setminus D$, then, because $D\subseteq Y$ and $Y$ is a subspace of $X$, the subspace closure formula would give
\[
x\in Y\cap \overline{D}^{\,X}=\overline{D}^{\,Y}=D,
\]
a contradiction.
Hence $x\notin Y$, so $x\in B\setminus Y$.
Let $U$ be an open neighborhood of $x$ in $X$.
Since $x\in \overline{D}^{\,X}$ and $x\notin D$, we have $U\cap D\neq\varnothing$.
If $U\cap D$ were finite, then, because $X$ is $T_1$, the set $U\cap D$ would be closed in $X$. Since $x\notin U\cap D$, there would exist an open neighborhood $V\subseteq U$ of $x$ with
\[
V\cap (U\cap D)=\varnothing.
\]
Because $V\subseteq U$, this would imply $V\cap D=\varnothing$, contradicting $x\in \overline{D}^{\,X}$.
Thus $U\cap D$ is infinite.
For each $d\in U\cap D$, since $D$ is discrete in $Y$, there exists an open set $O_d$ in $Y$ such that
\[
d\in O_d
\quad\text{and}\quad
O_d\cap D=\{d\}.
\]
Then
\[
(U\cap Y)\cap O_d
\]
is open in $U\cap Y$, contains $d$, and satisfies
\[
\bigl((U\cap Y)\cap O_d\bigr)\cap (U\cap D)=\{d\}.
\]
Thus $U\cap D$ is discrete in $U\cap Y$.
Also,
\[
U\cap D=(U\cap Y)\cap D.
\]
Since $D$ is closed in $Y$ and $U\cap Y$ is a subspace of $Y$, the set $(U\cap Y)\cap D=U\cap D$ is closed in $U\cap Y$.
Thus $U\cap D$ is a countable closed discrete subspace of $U\cap Y$.
Hence $U\cap Y$ cannot be countably compact.
\end{proof}

\begin{proposition}\label{prop_countable_cd_closure_D}
Let $X$ be a charming space with a Lindel\"of $\Sigma$ kernel $Y$.
If $D\subseteq Y$ is nonempty, countable and closed discrete in $Y$, then
\[
C=\overline{D}^{\,X}
\]
is a closed charming subspace of $X$ with countable open dense Lindel\"of $\Sigma$ kernel $D$.
In particular, $C\setminus D$ is Lindel\"of $\Sigma$ and $C$ is a \(D\)-space.
\end{proposition}

\begin{proof}
Since $D\neq\varnothing$, the space $C$ is nonempty. As a closed subspace of the Lindel\"of space $X$, it is Lindel\"of.
Because $D$ is countable, it is Lindel\"of $\Sigma$.
By definition, $D$ is dense in $C$.

We claim that
\[
C\cap Y=D.
\]
The inclusion $D\subseteq C\cap Y$ is clear.
Conversely, let $y\in Y\setminus D$.
Since $D$ is closed in the subspace $Y$, there exists an open set $U_y$ in $Y$ such that
\[
y\in U_y
\quad\text{and}\quad
U_y\cap D=\varnothing.
\]
Write
\[
U_y=O_y\cap Y
\]
for some open set $O_y$ in $X$.
Because $D\subseteq Y$, we obtain
\[
O_y\cap D=U_y\cap D=\varnothing.
\]
Hence
\[
y\notin \overline{D}^{\,X}=C.
\]
So $C\cap Y=D$.

We next show that $D$ is a Lindel\"of $\Sigma$ kernel of $C$.
Let $U$ be an open neighborhood of $D$ in $C$.
Choose an open set $W$ in $X$ such that
\[
U=C\cap W.
\]
Put
\[
O=W\cup (X\setminus C).
\]
Then $O$ is open in $X$. Because $D\subseteq U=C\cap W$, we have $D\subseteq W$. Also,
\[
Y\cap C=D.
\]
Hence every point of $Y\cap C$ lies in $W$, while every point of $Y\setminus C$ lies in $X\setminus C$. Therefore
\[
Y=(Y\cap C)\cup (Y\setminus C)=D\cup (Y\setminus C)\subseteq O.
\]
Since $Y$ is a Lindel\"of $\Sigma$ kernel of $X$, the complement
\[
X\setminus O
\]
is Lindel\"of $\Sigma$.
On the other hand,
\[
C\setminus U
=
C\cap (X\setminus W)
=
X\setminus O.
\]
Hence $C\setminus U$ is Lindel\"of $\Sigma$.
Therefore $D$ is a Lindel\"of $\Sigma$ kernel of $C$, and $C$ is charming by Definition~\ref{def_charming}.

We next show that $D$ is open in $C$.
Fix $d\in D$.
Because $D$ is discrete in the subspace $Y$, there exists an open set $U_d$ in $Y$ such that
\[
d\in U_d
\quad\text{and}\quad
U_d\cap D=\{d\}.
\]
Write
\[
U_d=O_d\cap Y
\]
for some open set $O_d$ in $X$.
Since $D\subseteq Y$, we obtain
\[
O_d\cap D=U_d\cap D=\{d\}.
\]
We claim that
\[
O_d\cap C=\{d\}.
\]
Indeed, suppose that some point $x\in (O_d\cap C)\setminus\{d\}$ exists.
Since $X$ is Hausdorff, there is an open set $V_x$ in $X$ such that
\[
x\in V_x
\quad\text{and}\quad
d\notin V_x.
\]
Then $V_x\cap O_d$ is an open neighborhood of $x$ in $X$.
Since
\[
x\in C=\overline{D}^{\,X},
\]
we obtain
\[
(V_x\cap O_d)\cap D\neq\varnothing.
\]
But $O_d\cap D=\{d\}$ and $d\notin V_x$, so $(V_x\cap O_d)\cap D=\varnothing$, a contradiction.
Thus $O_d\cap C=\{d\}$.
So each singleton $\{d\}$ is open in $C$, and therefore $D$ is open in $C$.

Since $D$ is open in $C$, we may apply Definition~\ref{def_charming} with $U=D$. It follows that $C\setminus D$ is Lindel\"of $\Sigma$.
Since $D$ is dense in $C$, the closure of the kernel $D$ in $C$ is all of $C$.
Since $C\setminus D$ is closed in $C$, it is in particular an $F_\sigma$ subset of $C$.
Proposition~\ref{prop_boundary_easycriteria}\textup{(ii)}, applied to the charming space $C$ with kernel $D$, then shows that $C$ is a \(D\)-space.
\end{proof}

\begin{corollary}\label{cor_badpoint_inside_countable_cd_closure}
Let $X$ be a charming space with a Lindel\"of $\Sigma$ kernel $Y$, and put $B=\overline{Y}^{\,X}$.
Assume that $x\in B\setminus Y$ and that there exists a nonempty countable set $D\subseteq Y$ such that $D$ is closed discrete in $Y$ and
\[
x\in \overline{D}^{\,X}.
\]
Then the closed hull
\[
C=\overline{D}^{\,X}
\]
is a charming \(D\)-space with countable open dense kernel $D$, the point $x$ belongs to $C\setminus D$, and the boundary part $C\setminus D$ is Lindel\"of $\Sigma$.
In particular, every individual boundary point witnessed by a countable closed discrete subset of $Y$ lies inside a closed \(D\)-subspace with countable open dense kernel.
\end{corollary}

\begin{proof}
Because $D\subseteq Y$ and $x\notin Y$, we have $x\notin D$.
So $x\in C\setminus D$.
The remaining assertions follow directly from Proposition~\ref{prop_countable_cd_closure_D}.
\end{proof}

\subsection{Reduction to a closed core}

We now isolate the boundary points at which the local countable-compactness condition from Proposition~\ref{prop_local_countably_compact_boundary} fails.

\begin{definition}\label{def_cc_bad_boundary}
Let $X$ be a charming space with a Lindel\"of $\Sigma$ kernel $Y$, and put $B=\overline{Y}^{\,X}$.
Set
\[
\mathsf{Obs}_{cc}(Y,B)=
\bigl\{x\in B\setminus Y :
\begin{aligned}[t]
&\text{for every open neighborhood } U \text{ of } x \text{ in } X,\\
&U\cap Y \text{ is not countably compact}
\end{aligned}
\bigr\},
\]
and call it the \emph{countably compact obstruction set}. Equivalently, the same set is obtained by using open neighborhoods of $x$ in $B$, because $B$ has the subspace topology inherited from $X$ and $Y\subseteq B$.
\end{definition}

\begin{lemma}\label{lem_cc_bad_boundary_closed}
Let $X$ be a charming space with a Lindel\"of $\Sigma$ kernel $Y$, and put $B=\overline{Y}^{\,X}$.
Then $\mathsf{Obs}_{cc}(Y,B)$ is closed in $B\setminus Y$.
If $X$ is not a \(D\)-space, then $\mathsf{Obs}_{cc}(Y,B)$ is nonempty.
\end{lemma}

\begin{proof}
Put $G=(B\setminus Y)\setminus \mathsf{Obs}_{cc}(Y,B)$.
Let $x\in G$. By definition there exists an open neighborhood $U_x$ of $x$ in $X$ such that
\[
U_x\cap Y
\]
is countably compact.
For every point $z\in U_x\cap (B\setminus Y)$, the same open set $U_x$ is a neighborhood of $z$ in $X$ and still meets $Y$ in a countably compact set, so $z$ also belongs to $G$. Thus
\[
U_x\cap (B\setminus Y)\subseteq G.
\]
So $G$ is open in the subspace $B\setminus Y$, and $\mathsf{Obs}_{cc}(Y,B)$ is closed in $B\setminus Y$.

If $X$ is not a \(D\)-space, Proposition~\ref{prop_counterexample_no_local_cc} gives a point $x\in B\setminus Y$ such that for every open neighborhood $U$ of $x$ in $X$, the set $U\cap Y$ is not countably compact.
Such a point lies in $\mathsf{Obs}_{cc}(Y,B)$, so the obstruction set is nonempty.
\end{proof}

\begin{corollary}\label{cor_obscc_closed_layer_small}
Let $X$ be a charming space with a Lindel\"of $\Sigma$ kernel $Y$, and put $B=\overline{Y}^{\,X}$.
Assume that there exists a closed subset $S$ of $B$ such that $S\subseteq B\setminus Y$ and
\[
\mathsf{Obs}_{cc}(Y,B)\subseteq S
\quad\text{and}\quad
\covK(S)<\mathfrak d.
\]
Then $X$ is a \(D\)-space.
\end{corollary}

\begin{proof}
If $x\in (B\setminus Y)\setminus S$, then $x\notin \mathsf{Obs}_{cc}(Y,B)$.
By Definition~\ref{def_cc_bad_boundary}, there exists an open neighborhood $U_x$ of $x$ in $X$ such that $U_x\cap Y$ is countably compact.
The conclusion follows from Theorem~\ref{thm_middlelayer_cc_small}.
\end{proof}

\begin{lemma}\label{lem_closed_hereditary_D}
Every closed subspace of a \(D\)-space is a \(D\)-space.
\end{lemma}

\begin{proof}
Let $F$ be a closed subspace of a \(D\)-space $X$, and let $\varphi$ be a neighborhood assignment on $F$.
For each $x\in F$, choose an open set $U_x\subseteq X$ such that $\varphi(x)=U_x\cap F$.
Define a neighborhood assignment $\psi$ on $X$ by
\[
\psi(x)=
\begin{cases}
U_x\cup(X\setminus F), & x\in F,\\
X\setminus F, & x\in X\setminus F.
\end{cases}
\]
Choose a closed discrete set $D\subseteq X$ such that
\[
X=\bigcup_{d\in D}\psi(d).
\]
No point of $F$ belongs to $\psi(d)$ when $d\in X\setminus F$.
Hence $D\cap F$ is closed discrete in $F$ and
\[
F=\bigcup_{d\in D\cap F}\varphi(d).
\]
\end{proof}

\begin{definition}\label{def_cc_core}
Let $X$ be a charming space with a Lindel\"of $\Sigma$ kernel $Y$, and put $B=\overline{Y}^{\,X}$.
Set
\[
\mathsf{Core}_{cc}(Y,B)=\overline{\mathsf{Obs}_{cc}(Y,B)}^{\,B},
\]
and call it the \emph{closed core}.
\end{definition}

\begin{lemma}\label{lem_cc_core_structure}
Let $X$ be a charming space with a Lindel\"of $\Sigma$ kernel $Y$, and put $B=\overline{Y}^{\,X}$.
Put
\[
H=\mathsf{Core}_{cc}(Y,B)
\quad\text{and}\quad
Y_H=H\cap Y.
\]
Then
\[
H\cap (B\setminus Y)=\mathsf{Obs}_{cc}(Y,B).
\]
Consequently $Y_H$ is closed in $Y$, hence Lindel\"of $\Sigma$.
\end{lemma}

\begin{proof}
By Lemma~\ref{lem_cc_bad_boundary_closed}, the set $\mathsf{Obs}_{cc}(Y,B)$ is closed in the subspace $B\setminus Y$.
So there exists a closed set $F\subseteq B$ such that
\[
\mathsf{Obs}_{cc}(Y,B)=F\cap(B\setminus Y).
\]
Because $\mathsf{Obs}_{cc}(Y,B)\subseteq F$, we have
\[
H=\overline{\mathsf{Obs}_{cc}(Y,B)}^{\,B}\subseteq F.
\]
Hence
\[
H\cap(B\setminus Y)\subseteq F\cap(B\setminus Y)=\mathsf{Obs}_{cc}(Y,B).
\]
Since $\mathsf{Obs}_{cc}(Y,B)\subseteq H$ by definition, we obtain
\[
H\cap(B\setminus Y)=\mathsf{Obs}_{cc}(Y,B).
\]
Since $H$ is closed in $B$, the set $Y_H=H\cap Y$ is closed in $Y$.
Consequently $Y_H$ is Lindel\"of $\Sigma$.
\end{proof}

\begin{theorem}\label{thm_cc_core_exact_reduction}
Let $X$ be a charming space with a Lindel\"of $\Sigma$ kernel $Y$, and put $B=\overline{Y}^{\,X}$.
Let
\[
H=\mathsf{Core}_{cc}(Y,B).
\]
Then for every open set $U\subseteq B$ with $H\subseteq U$, the closed complement $B\setminus U$ is a \(D\)-space.
Consequently, $B$ is a \(D\)-space if and only if $H$ is a \(D\)-space. The same equivalence holds for $X$ and $H$.
\end{theorem}

\begin{proof}
Since $B$ is closed in the Lindel\"of space $X$, $B$ is Lindel\"of.
Let $U\subseteq B$ be open and satisfy $H\subseteq U$.
Put
\[
C=B\setminus U
\quad\text{and}\quad
Y_C=Y\cap C.
\]
We show that $C$ is a \(D\)-space.
By Lemma~\ref{lem_closure_charming}, the space $B$ is charming with dense Lindel\"of $\Sigma$ kernel $Y$.
If $Y\subseteq U$, then $C=B\setminus U$ is Lindel\"of $\Sigma$ by Definition~\ref{def_charming} applied to the charming space $B$. By \cite[p.~494, Theorem]{Buzyakova2002}, $C$ is a \(D\)-space.

Suppose instead that $Y\not\subseteq U$, so that $Y_C\neq\varnothing$.
Because $H\subseteq U$ and $C=B\setminus U$, we have
\[
C\cap H=\varnothing.
\]
Since $C$ is closed in $B$, the set $Y_C=Y\cap C$ is closed in the Lindel\"of $\Sigma$ space $Y$, hence Lindel\"of $\Sigma$.
The space $C$ is closed in the Lindel\"of space $B$, so $C$ is Lindel\"of.
We show that $Y_C$ is a Lindel\"of $\Sigma$ kernel of $C$.
Let $V$ be an open neighborhood of $Y_C$ in $C$, and write $V=W\cap C$ for some open set $W$ in $B$.
Since $Y_C\subseteq W$ and $Y\setminus C\subseteq U$, we have $Y\subseteq W\cup U$.
Because $B$ is charming with kernel $Y$ and $W\cup U$ is an open neighborhood of $Y$ in $B$, the complement $B\setminus(W\cup U)$ is Lindel\"of $\Sigma$.
Now
\[
C\setminus V=(B\setminus U)\setminus(W\cap C)=B\setminus(U\cup W),
\]
so $C\setminus V$ is Lindel\"of $\Sigma$.
Thus $Y_C$ is a Lindel\"of $\Sigma$ kernel of $C$, and Definition~\ref{def_charming} shows that $C$ is charming.

Let $x\in C\setminus Y_C$.
Then $x\in C\subseteq B$, and since $x\notin Y_C=Y\cap C$ while $x\in C$, we have $x\notin Y$, so $x\in C\cap (B\setminus Y)$. Since $C\cap H=\varnothing$, Lemma~\ref{lem_cc_core_structure} gives
\[
x\notin \mathsf{Obs}_{cc}(Y,B).
\]
By Definition~\ref{def_cc_bad_boundary}, there exists an open neighborhood $W_x$ of $x$ in $X$ such that
\[
W_x\cap Y
\]
is countably compact.
Put
\[
V_x=W_x\cap C.
\]
Then $V_x$ is an open neighborhood of $x$ in $C$, and
\[
V_x\cap Y_C=W_x\cap Y_C.
\]
Because $Y_C=Y\cap C$ is closed in $Y$, the subspace
\[
W_x\cap Y_C=(W_x\cap Y)\cap Y_C
\]
is closed in the countably compact space $W_x\cap Y$.
So $V_x\cap Y_C=W_x\cap Y_C$ is countably compact.
Thus, for every $x\in C\setminus Y_C$, there exists an open neighborhood $V_x$ of $x$ in $C$ such that $V_x\cap Y_C$ is countably compact. In particular, this holds for every $x\in\overline{Y_C}^{\,C}\setminus Y_C\subseteq C\setminus Y_C$.
Therefore Proposition~\ref{prop_local_countably_compact_boundary}, applied to the charming space $C$ with Lindel\"of $\Sigma$ kernel $Y_C$ and closure $\overline{Y_C}^{\,C}$, shows that $C$ is a \(D\)-space.

Thus, whenever $U\subseteq B$ is open and $H\subseteq U$, the closed set $B\setminus U$ is a \(D\)-space.
If $H$ is a \(D\)-space, Lemma~\ref{lem_glue}, applied to the closed subspace $H$ of $B$, gives that $B$ is a \(D\)-space, while Lemma~\ref{lem_closed_hereditary_D} gives the converse. Therefore
\[
B\text{ is a \(D\)-space}
\quad\Longleftrightarrow\quad
H\text{ is a \(D\)-space}.
\]
Since $H$ is closed in $X$, $X$ being a \(D\)-space implies that $H$ is a \(D\)-space.
Conversely, if $H$ is a \(D\)-space, then the equivalence above gives that $B$ is a \(D\)-space, and Proposition~\ref{prop_reduction_charming} shows that $X$ is a \(D\)-space.
So the same equivalence holds for $X$ and $H$.
\end{proof}

\begin{proposition}\label{prop_cc_core_small}
Let $X$ be a charming space with a Lindel\"of $\Sigma$ kernel $Y$, and put $B=\overline{Y}^{\,X}$.
If
\[
\covK\bigl(\mathsf{Core}_{cc}(Y,B)\bigr)<\mathfrak d,
\]
then $X$ is a \(D\)-space.
\end{proposition}

\begin{proof}
Put
\[
H=\mathsf{Core}_{cc}(Y,B).
\]
The closed subspace $H$ of $X$ is Lindel\"of.
Since $\covK(H)<\mathfrak d$, Lemma~\ref{lem_lt_d_compact_Menger} shows that $H$ is Menger.
By Theorem~\ref{thm_MengerD}, the space $H$ is a \(D\)-space.
Theorem~\ref{thm_cc_core_exact_reduction} then gives that $X$ is a \(D\)-space.
\end{proof}

\begin{proposition}\label{prop_cc_core_Gdelta}
Let $X$ be a charming space with a Lindel\"of $\Sigma$ kernel $Y$, and put $B=\overline{Y}^{\,X}$.
Put
\[
H=\mathsf{Core}_{cc}(Y,B)
\quad\text{and}\quad
Y_H=H\cap Y.
\]
If $Y_H$ is a $G_\delta$ subset of $H$, then $X$ is a \(D\)-space.
\end{proposition}

\begin{proof}
Since $H$ is closed in the Lindel\"of space $B$, the space $H$ is Lindel\"of.
If $Y_H=\varnothing$, then $B\setminus H$ is an open neighborhood of $Y$ in $B$. Since $Y$ is a kernel of the charming space $B$, Definition~\ref{def_charming} therefore implies that $H$ is Lindel\"of $\Sigma$. By \cite[p.~494, Theorem]{Buzyakova2002}, $H$ is a \(D\)-space, and Theorem~\ref{thm_cc_core_exact_reduction} gives the conclusion.
Assume from now on that $Y_H\neq\varnothing$.
Because $Y_H$ is closed in the Lindel\"of $\Sigma$ space $Y$, the set $Y_H$ is Lindel\"of $\Sigma$.

We show that $Y_H$ is a Lindel\"of $\Sigma$ kernel of $H$.
Let $V$ be an open neighborhood of $Y_H$ in $H$, and write
\[
V=W\cap H
\]
for some open set $W$ in $B$.
Put
\[
O=W\cup (B\setminus H).
\]
Then $O$ is open in $B$, and
\[
Y=(Y\cap H)\cup (Y\setminus H)=Y_H\cup (Y\setminus H)\subseteq O.
\]
Since $B$ is charming with kernel $Y$, the complement
\[
B\setminus O
\]
is Lindel\"of $\Sigma$.
On the other hand,
\[
H\setminus V
=
H\cap (B\setminus W)
=
B\setminus O.
\]
So $H\setminus V$ is Lindel\"of $\Sigma$.
Thus $Y_H$ is a Lindel\"of $\Sigma$ kernel of $H$, and Definition~\ref{def_charming} shows that $H$ is charming.

Put
\[
B_H=\overline{Y_H}^{\,H}.
\]
Since $Y_H$ is a $G_\delta$ subset of $H$ and $B_H\subseteq H$, the set $Y_H$ is also a $G_\delta$ subset of $B_H$.
Proposition~\ref{prop_Gdelta_kernel}, applied to the charming space $H$ with kernel $Y_H$ and closure $B_H$, now shows that $H$ is a \(D\)-space.
Theorem~\ref{thm_cc_core_exact_reduction} then gives the conclusion.
\end{proof}

\begin{proposition}\label{prop_counterexample_cc_core_exact}
Let $X$ be a charming space that is not a \(D\)-space.
Let $Y$ be a Lindel\"of $\Sigma$ kernel of $X$, put $B=\overline{Y}^{\,X}$, and set
\[
H=\mathsf{Core}_{cc}(Y,B).
\]
Then $H$ is not a \(D\)-space,
\[
H\cap (B\setminus Y)=\mathsf{Obs}_{cc}(Y,B),
\]
and
\[
\covK(H)\ge \mathfrak d.
\]
In particular, $H$ is not $\sigma$-compact.
\end{proposition}

\begin{proof}
By Theorem~\ref{thm_cc_core_exact_reduction}, the space $H$ is not a \(D\)-space.
Lemma~\ref{lem_cc_core_structure} gives
\[
H\cap (B\setminus Y)=\mathsf{Obs}_{cc}(Y,B).
\]
If $\covK(H)<\mathfrak d$, then Proposition~\ref{prop_cc_core_small} would imply that $X$ is a \(D\)-space, contrary to the assumption.
Thus $\covK(H)\ge \mathfrak d$.
If $H$ were $\sigma$-compact, then $\covK(H)\le \omega<\mathfrak d$, a contradiction.
\end{proof}

\begin{proposition}\label{prop_counterexample_core_not_countably_generated}
Let $X$ be a charming space that is not a \(D\)-space.
Let $Y$ be a Lindel\"of $\Sigma$ kernel of $X$, put $B=\overline{Y}^{\,X}$, and set
\[
H=\mathsf{Core}_{cc}(Y,B).
\]
Then no countable closed discrete subset of $H\cap Y$ is dense in $H$.
Equivalently, $H$ is not the closure in $H$ of any countable closed discrete subset of $H\cap Y$.
\end{proposition}

\begin{proof}
If some countable closed discrete set $D\subseteq H\cap Y$ were dense in $H$, then
\[
H=\overline{D}^{\,H}=H\cap \overline{D}^{\,X},
\]
by the subspace closure formula. In particular,
\[
H\subseteq \overline{D}^{\,X}.
\]
By Lemma~\ref{lem_cc_core_structure}, the set $H\cap Y$ is closed in $Y$. Since $D$ is closed in the subspace $H\cap Y$, it is closed in $Y$.
For each $d\in D$, because $D$ is discrete in the subspace $H\cap Y$, there exists an open set $U_d$ in $H\cap Y$ such that
\[
d\in U_d
\quad\text{and}\quad
U_d\cap D=\{d\}.
\]
Write
\[
U_d=(H\cap Y)\cap O_d
\]
for some open set $O_d$ in $Y$.
Since $D\subseteq H\cap Y$, we obtain
\[
O_d\cap D=U_d\cap D=\{d\}.
\]
Thus $D$ is discrete in $Y$, and hence closed discrete in $Y$.
Because $H$ is not a \(D\)-space, it is nonempty. Since $D$ is dense in $H$, the set $D$ is nonempty as well.
Proposition~\ref{prop_countable_cd_closure_D}, applied in the original charming space $X$, shows that $\overline{D}^{\,X}$ is a closed charming \(D\)-subspace of $X$.
Since $H$ is closed in $\overline{D}^{\,X}$, Lemma~\ref{lem_closed_hereditary_D} would then force $H$ to be a \(D\)-space.
This contradicts Proposition~\ref{prop_counterexample_cc_core_exact}.
\end{proof}

\begin{corollary}\label{cor_counterexample_obscc_not_small_closed_layer}
Let $X$ be a charming space that is not a \(D\)-space.
Let $Y$ be a Lindel\"of $\Sigma$ kernel of $X$ and put $B=\overline{Y}^{\,X}$.
If $S\subseteq B\setminus Y$ is closed in $B$ and
\[
\mathsf{Obs}_{cc}(Y,B)\subseteq S,
\]
then
\[
\covK(S)\ge \mathfrak d.
\]
In particular, no such closed set $S$ is $\sigma$-compact.
\end{corollary}

\begin{proof}
If $\covK(S)<\mathfrak d$, then Corollary~\ref{cor_obscc_closed_layer_small} would imply that $X$ is a \(D\)-space, contrary to the assumption.
Thus $\covK(S)\ge \mathfrak d$.
If $S$ were $\sigma$-compact, then $\covK(S)\le \omega<\mathfrak d$, again a contradiction.
\end{proof}

\subsection{The remaining question for the closed core}

The reductions above force the following properties on any counterexample.
If $X$ is charming but not a \(D\)-space, with kernel $Y$, closure $B=\overline{Y}^{\,X}$, and core
\[
H=\mathsf{Core}_{cc}(Y,B),
\]
then one has
\begin{enumerate}[label=\textup{(\roman*)}]
\item $H$ is not a \(D\)-space and $\covK(H)\ge \mathfrak d$.
\item
\[
H\cap (B\setminus Y)=\mathsf{Obs}_{cc}(Y,B).
\]
\item No countable closed discrete subset of $H\cap Y$ is dense in $H$. Equivalently, $H$ is not the closure in $H$ of any countable closed discrete subset of $H\cap Y$.
\item If $S\subseteq B\setminus Y$ is closed in $B$ and
\[
\mathsf{Obs}_{cc}(Y,B)\subseteq S,
\]
then
\[
\covK(S)\ge \mathfrak d.
\]
In particular, no such closed set $S$ is $\sigma$-compact.
\end{enumerate}
After these reductions, the problem is concentrated on the closed subspace $H$. Theorem~\ref{thm_equiv_dense_extension} reduces the general question to the dense-kernel case. Theorem~\ref{thm_middlelayer_cc_small} handles a closed exceptional layer of compact covering number below $\mathfrak d$, provided that local countable compactness along $Y$ holds at every remaining boundary point. By Definition~\ref{def_cc_core} and Lemma~\ref{lem_cc_core_structure},
\[
H=\mathsf{Core}_{cc}(Y,B)=\overline{\mathsf{Obs}_{cc}(Y,B)}^{\,B}
\quad\text{and}\quad
H\cap (B\setminus Y)=\mathsf{Obs}_{cc}(Y,B).
\]
Theorem~\ref{thm_cc_core_exact_reduction} shows that $X$ is a \(D\)-space if and only if $H$ is. On the other hand, Proposition~\ref{prop_countable_cd_closure_D} shows that the closure in $X$ of every nonempty countable closed discrete subset of $Y$ is already a charming \(D\)-space. Together, Propositions~\ref{prop_counterexample_cc_core_exact} and~\ref{prop_counterexample_core_not_countably_generated} show that every counterexample core satisfies $\covK(H)\ge\mathfrak d$ and cannot be the closure in $H$ of a countable closed discrete subset of $H\cap Y$. These facts suggest the following conjecture.

\begin{conjecture}\label{conj_select_good_core}
For every charming space $X$, there exists a Lindel\"of $\Sigma$ kernel $Y$ of $X$ such that, for $B=\overline{Y}^{\,X}$ and $H=\mathsf{Core}_{cc}(Y,B)$, the set $H\cap Y$ is dense in $H$ and $G_\delta$ in $H$.
\end{conjecture}

Under this conjecture, Proposition~\ref{prop_cc_core_Gdelta} implies that every charming space is a \(D\)-space. For this deduction, only the $G_\delta$ condition is used. The density condition is an additional structural requirement, included to ensure that the surviving kernel trace $H\cap Y$ remains dense in the closed core.

\subsection{Applications to remainders}

For a Hausdorff compactification $bX$ of a space $X$, we call $bX\setminus X$ the remainder. Results on the structure of remainders go back to Henriksen and Isbell \cite{HenriksenIsbell1958}. The next two lemmas give elementary criteria for remainders. The first is another application of the $G_\delta$-kernel argument. We then apply known remainder theorems to topological groups and compactly-fibered coset spaces. None of these results is used in the preceding reductions.

\begin{lemma}\label{lem_remainder_trace_countchar}
Let $X$ be a Tychonoff space and let $bX$ be a Hausdorff compactification.
Put $R=bX\setminus X$ and assume that $R$ is charming.
Let $Y$ be a Lindel\"of $\Sigma$ kernel of $R$.
If there exists a compact set $K\subseteq bX$ of countable character in $bX$ such that $Y=K\cap R$, then $R$ is a \(D\)-space.
\end{lemma}

\begin{proof}
Let $\{W_n:n\in\omega\}$ be a countable open neighborhood base of $K$ in $bX$. Since $K\subseteq W_n$ for every $n\in\omega$, we have
\[
K\subseteq\bigcap_{n\in\omega} W_n.
\]
Conversely, let $z\in bX\setminus K$. Since $bX$ is Hausdorff, the set $bX\setminus\{z\}$ is an open neighborhood of $K$. Hence $W_n\subseteq bX\setminus\{z\}$ for some $n\in\omega$, and therefore $z\notin W_n$. It follows that
\[
K=\bigcap_{n\in\omega} W_n,
\]
so $K$ is a $G_\delta$ subset of $bX$.
Since $Y=K\cap R$, we obtain
\[
Y=R\cap K=\bigcap_{n\in\omega}(R\cap W_n).
\]
Each set $R\cap W_n$ is open in $R$, so $Y$ is a $G_\delta$ subset of $R$.
Because $\overline{Y}^{\,R}$ is a closed subspace of $R$ and each set $R\cap W_n$ is open in $R$, each set
\[
\overline{Y}^{\,R}\cap (R\cap W_n)
\]
is open in $\overline{Y}^{\,R}$. Moreover,
\[
Y=\overline{Y}^{\,R}\cap Y
=\overline{Y}^{\,R}\cap \bigcap_{n\in\omega}(R\cap W_n)
=\bigcap_{n\in\omega}\bigl(\overline{Y}^{\,R}\cap (R\cap W_n)\bigr).
\]
Thus $Y$ is a $G_\delta$ subset of $\overline{Y}^{\,R}$.
Now Proposition~\ref{prop_Gdelta_kernel}, applied to the charming space $R$ with kernel $Y$, gives the conclusion.
\end{proof}

\begin{lemma}\label{lem_remainder_sigmacompact_test}
Let $X$ be a Tychonoff space and let $bX$ be a Hausdorff compactification.
Put $R=bX\setminus X$.
Assume that $K\subseteq R$ is compact and has a countable neighborhood base in $bX$. Suppose that $(U_n)_{n\in\omega}$ is a decreasing neighborhood base of $K$ in $bX$ with $K=\bigcap_{n\in\omega}U_n$. Put $V_n=U_n\cap R$, and assume that $R\setminus V_n$ is $\sigma$-compact for every $n\in\omega$.
Then $R$ is $\sigma$-compact.
In particular, $R$ is a \(D\)-space.
\end{lemma}

\begin{proof}
Since $K=\bigcap_{n\in\omega}U_n$ and $K\subseteq R$, we have
\[
K=\bigcap_{n\in\omega}V_n.
\]
Moreover,
\[
R\setminus V_n=R\setminus (U_n\cap R)=R\cap (bX\setminus U_n).
\]
By hypothesis, $R\setminus V_n$ is $\sigma$-compact for every $n\in\omega$.
If $x\in R\setminus K$, then $x\notin V_n$ for some $n\in\omega$, because $K=\bigcap_{n\in\omega}V_n$.
Therefore
\[
R=K\cup\bigcup_{n\in\omega}(R\setminus V_n).
\]
For each $n\in\omega$, write
\[
R\setminus V_n=\bigcup_{m\in\omega} C_{n,m}
\]
with each $C_{n,m}$ compact.
Then
\[
R=K\cup\bigcup_{n,m\in\omega} C_{n,m},
\]
so $R$ is $\sigma$-compact. By \cite[p.~2576, after Proposition~1.2]{KubisOkunevSzeptycki2006}, $R$ is Lindel\"of $\Sigma$. By \cite[p.~494, Theorem]{Buzyakova2002}, $R$ is a \(D\)-space.
\end{proof}

\begin{corollary}\label{cor_remainder_compact_core_easy}
Let $X$ be a Tychonoff space and let $bX$ be a Hausdorff compactification.
Put $R=bX\setminus X$.
Assume that there exists a compact set $K\subseteq R$ that admits a countable neighborhood base in $bX$ such that $R\setminus K$ is $\sigma$-compact.
Then $R$ is $\sigma$-compact.
In particular, $R$ is a \(D\)-space.
\end{corollary}

\begin{proof}
Fix a decreasing neighborhood base $(U_n)_{n\in\omega}$ of $K$ in $bX$ such that $K=\bigcap_{n\in\omega}U_n$, and put
\[
V_n=U_n\cap R
\]
for each $n\in\omega$. Then
\[
R\setminus V_n=R\setminus (U_n\cap R)=R\cap (bX\setminus U_n),
\]
so $R\setminus V_n$ is closed in $R$.
Since $K\subseteq U_n$, we have $R\setminus V_n\subseteq R\setminus K$, and therefore
\[
R\setminus V_n=(R\setminus V_n)\cap (R\setminus K).
\]
Hence $R\setminus V_n$ is closed in the subspace $R\setminus K$.
Write
\[
R\setminus K=\bigcup_{m\in\omega} C_m
\]
with each $C_m$ compact.
Then
\[
R\setminus V_n=(R\setminus V_n)\cap (R\setminus K)=\bigcup_{m\in\omega}\bigl(C_m\cap (R\setminus V_n)\bigr),
\]
and each $C_m\cap (R\setminus V_n)$ is compact, because $C_m\subseteq R\setminus K$ and $R\setminus V_n$ is closed in the subspace $R\setminus K$, so $C_m\cap (R\setminus V_n)$ is closed in $C_m$.
Hence $R\setminus V_n$ is $\sigma$-compact for every $n\in\omega$.
Lemma~\ref{lem_remainder_sigmacompact_test}, applied to the compactification $bX$ and the remainder $R=bX\setminus X$, now gives the conclusion.
\end{proof}

\begin{theorem}\label{thm_remD_PengHu}
\cite[Theorem~2.29]{PengHu2024} If $X$ is a paracompact $p$-space, then for every Hausdorff compactification $bX$ of $X$, the remainder $bX\setminus X$ is a \(D\)-space.
\end{theorem}

Since metrizable spaces are paracompact $p$-spaces, Theorem~\ref{thm_remD_PengHu} implies that every remainder of a metrizable space in a Hausdorff compactification is a \(D\)-space, and this applies in particular when the remainder is charming.

Let $G$ be a topological group, let $bG$ be a Hausdorff compactification of $G$, and put $R=bG\setminus G$. If $G$ is locally compact, then $G$ is open in every Hausdorff compactification $bG$, so $R=bG\setminus G$ is compact, and hence trivially a \(D\)-space. We may therefore assume $G$ is not locally compact.
Suppose further that $R$ is charming, and hence Lindel\"of. By \cite[Theorem~1.7]{Arhangelskii2009Study}, the group $G$ is a paracompact $p$-space. Theorem~\ref{thm_remD_PengHu} therefore shows that $R$ is a \(D\)-space.

Let $X$ be a compactly-fibered coset space, let $bX$ be a Hausdorff compactification of $X$, and put $R=bX\setminus X$. If $R$ is charming, then $R$ is Lindel\"of. By \cite[Corollary~2.30]{PengHu2024}, $R$ is a \(D\)-space.

\end{document}